\documentclass[11pt]{amsart}
\usepackage{geometry}                
\usepackage[parfill]{parskip}    
\usepackage{graphicx}
\usepackage{color}
\usepackage{amsmath,amscd}
\usepackage{amssymb}
\usepackage{epstopdf}
\usepackage[all]{xy}
\usepackage{eucal}
\usepackage{euler}
\usepackage{times}
\title{A cohomological characterisation of Yu's property A for metric spaces}

\def\EE#1#2{\text{$E^{#1}_{#2}$}}
\def\naturals{\mathbb{N}}

\def\integers{\mathbb{Z}}

\DeclareMathOperator{\supp}{Supp}

\def\HE{H_E}
\def\Ho{H_{E_0}}

\def\Epq{\mathcal E^{p,q}}
\def\E{\mathcal E}

\def\He{H_\E}
\def\HQ{H_Q}
\def\HQA{H_{QA}}
\def\HW{H_W}
\def\HWA{H_{WA}}
\def\HQW{H_{\sim}}

\def\EQ{\mathcal E_Q}
\def\EQA{\mathcal E_{QA}}
\def\EW{\mathcal E_W}
\def\EWA{\mathcal E_{WA}}

\def\XV{(X,\mathcal V)}

\def\J{\mathcal J}

\def\bx{\mathbf{x}}
\def\by{\mathbf{y}}
\def\bxy{(\bx,\by)}

\def\la{\langle}
\def\ra{\rangle}

\def \weakstar {\text{weak-*}\ }

\newcommand{\one}{\ell^1}

\newcommand{\bfy}{\mathbf{y}}
\newcommand{\bfx}{\mathbf{x}}

\newcommand{\calj}{\mathcal{J}}

\newcommand{\calv}{\mathcal{V}}

 \def \note#1{{ \color{red}#1}}
\newtheorem{thm}{Theorem}
\newtheorem*{mainthm}{Theorem \ref{MainTheorem1}}
\newtheorem*{mainthmtwo}{Theorem \ref{MainTheorem2}}
\newtheorem*{mainthmthree}{Theorem \ref{MainTheorem3}}
\newtheorem*{thm*}{Theorem}
\newtheorem{lemma}[thm]{Lemma}
\newtheorem{corollary}[thm]{Corollary}

\newtheorem{defn}{Definition}

\newtheorem{example}[defn]{Example}
\newtheorem{prop}[thm]{Proposition}
\newtheorem*{remark}{Remark}
\newcommand{\C}{\mathbb{C}}

\newcommand{\imp}{$\implies$}
\newcommand{\norm}[1]{{\|{#1}\|}}

\renewcommand{\Im}{\operatorname{Im}}

\newcommand{\skipit}[1]{\relax}

\def\note{\skipit}

\thanks{This research was partially supported by EPSRC grant  EP/F031947/1.}

\author{Jacek Brodzki}
\address{School of Mathematics, University of Southampton, Highfield, Southampton, SO17 1SH, England}
\email{J.Brodzki@soton.ac.uk}

\author{Graham A. Niblo}
\address{School of Mathematics, University of Southampton, Highfield, Southampton, SO17 1SH, England}
\email{G.A.Niblo@soton.ac.uk}

\author{Nick Wright}
\address{School of Mathematics, University of Southampton, Highfield, Southampton, SO17 1SH, England}
\email{N.J.Wright@soton.ac.uk}

\begin{document}

\begin{abstract}
Property A was introduced by Yu as a non-equivariant analogue of amenability. Nigel Higson posed the question of whether there is a homological characterisation of property A. In this paper we answer Higson's question affirmatively by constructing analogues of group cohomology and bounded cohomology, for a metric space $X$, and show that property A is equivalent to vanishing cohomology. Using these cohomology theories we also give a characterisation of property A in terms of the existence of an asymptotically invariant mean on the space.
\end{abstract}

\maketitle

A locally compact group $G$ is said to be \emph{amenable} if and only if it has an invariant mean \cite{vN}, that is, there exists $\mu\in \ell^\infty(G)^*$ such that $\mu(1)=1$ and $g\mu=\mu$ for all $g\in G$.

For a countable discrete group this is equivalent to the Reiter condition \cite{Reiter}:

For each $g\in G$ and each $n\in\naturals$ there is an element $f_n(g)\in \text{Prob}(G)$ of finite support with
\begin{enumerate}
\item $hf_n(g)=f_n(hg)$, and

\item for all $g_0,g_1$, \quad
$\|f_n(g_1) - f_n(g_0)\|_{\ell^1}{\rightarrow} 0 \text{ as }{n\rightarrow \infty}.$
\end{enumerate}

Ringrose and Johnson proved the following characterisation of amenability for a locally compact group, in terms of bounded cohomology.

\begin{thm*}
A group $G$ is amenable if and only if $H^q_b(G,V^*)=0$ for all $q\geq 1$ and all Banach $G$-modules $V$.
\end{thm*}

Moreover amenability is characterised by the vanishing of a specific class in $H^1_b(G,(\ell^\infty/\C)^*)$, namely the Johnson class $\calj(g_0,g_1)=\delta_{g_1}-\delta_{g_0}$. Here $\delta_g$ denotes the Dirac delta function supported at $g$ in $\one(G)$ which is included in $\one(G)^{**}\cong (\ell^\infty/\C)^*$ in the usual way. In \cite{ourpaper} we observed that in fact this characterisation also applies in classical (unbounded) group cohomology $H^1(G,(\ell^\infty/\C)^*)$.

Nigel Higson posed the question of whether there is a corresponding result for metric spaces and property A. Property A was introduced by Yu \cite{Yu}, as a non-equivariant analogue of amenability. The group action in the definition of amenability is replaced by a controlled support condition.

In this paper we answer Higson's question affirmatively by constructing analogues of group cohomology, bounded cohomology and the Johnson class, for a metric space $X$. These cohomologies have coefficients in a geometric module over $X$.

We use the following definition of property A. This is equivalent to Yu's original definition for spaces of bounded geometry \cite{HR}. 

\begin{defn}
\label{propAdef}
A metric space $X$ is said to have property A if for each $x\in X$ and each $n\in\naturals$ there is an element $f_n(x)\in \text{Prob}(X)$ and a sequence $S_n$ such that $\supp(f_n(x))\subseteq B_{S_n}(x)$ and for any $R\geq 0$
\[
\|f_n(x_1) - f_n(x_0)\|_{\ell^1}{\rightarrow} 0 \text{ as }{n\rightarrow \infty},
\]
uniformly on the set $\{ (x_0,x_1)\mid d(x_0,x_1)\leq R\}$.
\end{defn}

Note: This is an analogue of the Reiter condition for amenability. In Reiter's condition, uniform convergence, and the controlled support condition, follow from the pointwise convergence and finite support by equivariance. We view the convergence condition as asymptotic invariance in $x$ of the sequence $f_n(x)$.

Let $X$ be a metric space. An \emph{$X$-module} is a triple  $\mathcal V= (V, \|\cdot \|, Supp)$ where the pair $(V, \|\cdot \|)$ is a Banach space, and $\text{Supp}$ is a function from $V$ to the power set of $X$ satisfying the following axioms:

\begin{itemize}
\item $\text{Supp}(v)=\emptyset$ if  $v=0$,
\item $\text{Supp}(v+w)\subseteq  \text{Supp}(v)\cup \text{Supp}(w)$ for every $v,w\in V$,
\item $\text{Supp}(\lambda v)=\text{Supp}(v)$ for every $v\in V$ and every $\lambda\not=0$. 
\item if $v_n$ is a sequence converging to $v$ then $\supp(v)\subseteq \bigcup\limits_n \supp(v_n)$. \note {optional extra!}
\end{itemize}

\begin{example}
Let $V=\ell^1(X)$ equipped with the $\ell^1$- norm and for $f\in \ell^1(X)$ let $\text{Supp}(f)=\{x\in X\mid f(x)\not=0\}$.
\end{example}

Note that if $(W, \|\cdot\|)$ is a closed subspace of $(V, \|\cdot \|)$ then any function $\text{Supp}:V\mapsto 2^X$ satisfying conditions 1-3 above restricts to a function $\text{Supp}|_W$ so that $(W, \|\cdot\|, \text{Supp}|_W)$ is an $X$-module. We will consider the special case of the subspace $\ell^1_0(X)$ of  $\ell^1(X)$ consisting of functions $f$ such that  $\sum\limits_{x\in X}f(x)=0$, by analogy with the characterisation of amenability in terms of the vanishing of $\calj$.

If $X$ is equipped with a $G$ action for some group $G$, then we may also consider the notion of a $G$-equivariant $X$ module. This is an $X$-module $(V, \|\cdot\|, \text{Supp})$ equipped with an isometric action of $G$ such that $g\text{Supp}(v)=\text{Supp}(gv)$ for every $g\in G$ and every $v\in V$.

\skipit
{

In answer to a question of Nigel Higson we give several cohomological characterisations of Yu's property A via the introduction of suitable controlled cohomologies. The theories have equivariant and a non-equivariant versions, with the equivariant version characterising amenability for a group. We form a bicomplex $\E^{p,q}\XV$ associated to a space $X$ and an $X$-module $\calv$ (where these may or may not be equipped with a group action). We then proceed to `complete' the bicomplex in two different ways yielding two new cohomologies. We exploit the interaction between the two to give cohomological characterisations of property A in both theories, and to give another characterisation in terms of the existence of an asymptotically invariant mean on the space. These cohomology theories are in some sense analogous to group cohomology, and we also construct sub-complexes of both theories which correspond to bounded cohomology.


The equivariant theory is analogous to bounded cohomology in that it is zero if and only if the group is amenable, and both theories parallel Johnson's construction of bounded cohomology for a Banach module.
There is an alternative way to characterise amenability using Block and Weinberger's uniformly finite homology, \cite{BlockWeinberger}. In section \ref{exact} we explore the interaction between this homology and the classical bounded cohomology in the context of a discrete group and give purely cohomological proofs of Johnson's and Block and Weinberger's results in this context. \note{DO  WE WANT TO SAY SOMETHING HERE LIKE:  In a later paper we will generalise these ideas to study Yu's property A by introducing an analogous homology theory.}
}

\skipit{
In simplest terms the following diagram illustrates the components of the construction, with each row providing a cochain complex, with the final row providing the cohomology of our theory. 

\bigskip

\begin{center}
$\begin{CD}
\EE{m-1}{0}@>>> \EE{m}{0} @>>> \EE{m+1}{0}\\
@VVV @VVV @VVV \\
\EE{m-1}{ai}@>>> \EE{m}{ai} @>>> \EE{m+1}{ai}\\
@VVV @VVV @VVV \\
\EE{m-1}{ai}/ \EE{m-1} {0}@>>> \EE{m-1}{ai}/\EE{m}{0}@>>>\EE{m+1}{ai}/\EE{m-1}{0}\\
\end{CD}$
\end{center}

\bigskip
As we will later see this diagram is related to a family of natural bicomplexes which provide finer information about property A and provide a bridge between bounded cohomology and uniformly finite homology.

A curious feature of the setup is the introduction of the notion of an abstract support function for a Banach space to give it a ``module'' structure over a metric space. This is necessary since Yu's theory leverages control over supports to capture the local structure as well as the global coarse geometry. This stands in the tradition of controlled propagation, which, following the work of Roe and others, has done so much to illuminate K-theoretic aspects of coarse geometry.

}

We will construct two cohomology theories $\HQA^*\XV,\HWA^*\XV$. Both of these are analogues of bounded cohomology and they share many properties. The former has the virtue that it is straightforward to construct explicit cocycles, while the latter is more theoretical. For the $X$-module $\one_0(X)$, the cohomology groups both contain a Johnson class, and vanishing of this in either characterises property A, and moreover guarantees the vanishing of $\HQA^q\XV,\HWA^q\XV$ for all $q\geq 1$ and all $X$-modules $\calv$, Theorem \ref{MainTheorem2} and Theorem \ref{MainTheorem3}.

Vanishing of the Johnson class in $\HQA^1(X,\one_0(X))$ yields an asymptotically invariant Reiter sequence as in the above definition of property A, while its vanishing in $\HWA^1(X,\one_0(X))$ yields a new characterisation of property A in terms of the existence of an asymptotically invariant mean on the space $X$. As in the case of bounded cohomology and amenability, the vanishing theorem follows from an averaging argument, which utilises this asymptotically invariant mean.

A key step in the construction of $\HQA^*\XV,\HWA^*\XV$, is to construct two corresponding analogues of classical group cohomology $\HQ^*\XV,\HW^*\XV$. As in the case of a group, there are forgetful maps $\HQA^*\XV\to \HQ^*\XV,\HWA^*\XV\to\HW^*\XV$, and the vanishing of the images of the Johnson classes again characterises property A, Theorem \ref{MainTheorem1}.

As in \cite{ourpaper}, the equivalence of property A and vanishing of the Johnson elements in $\HQ^*\XV,\HW^*\XV$, is exhibited using a long exact sequence in cohomology, arising from a short exact sequence of coefficients:
$$0\to \one_0(X),\to \one(X)\to \C\to 0.$$
In each case, the Johnson class appears naturally as the image of the constant function 1 under the connecting map in the long exact sequence. The asymptotically invariant mean is a 0-cocycle for $\HW^*(X,\one(X))$, which is in the image of the forgetful map from $\HWA^*(X,\one(X))$.

There are also equivariant version of our cohomologies, when $X$ is a $G$-space, and $V$ a Banach $G$-module, for some group $G$. For convenience of exposition, we will assume throughout that we have such a $G$-action, allowing the possibility that $G=\{e\}$ as the non-equivariant case described above. In the case that $X=G$ is a countable discrete group with a proper left-invariant metric, equipped with the usual left-action, the cohomologies detect amenability.

The results in this paper are related in spirit to, but independent from the more recent results appearing in \cite{DN}. In their paper Douglas and Nowak prove that exactness of a finitely generated group $G$ (which by \cite{Oz} is equivalent to property A) implies vanishing of the bounded cohomology $H^q_b(G,V)$, for so-called Hopf $G$-modules of continuous linear operators with values in $\ell^\infty(G)$. They also give a characterisation of exactness in terms of the existence of an invariant conditional expectation on the group.


\skipit{
\begin{mainthm}
Let $X$ be a discrete  metric space. Then the following are equivalent:
\begin{enumerate}
\item $X$ has property $A$.
\item $[\mathbf 1_Q]\in \Im \pi_*$ in $\HQ^0(X, \C)$.
\item $D[\mathbf 1_Q]=0$ in $\HQ^1(X, \ell^1_0(X))$.
\item $D[\mathbf 1_W]=0$ in $\HW^1(X, \ell^1_0(X))$.
\item $[\mathbf 1_W]\in \Im \pi_*$ in $\HW^0(X, \C)$.
\item $X$ admits an asymptotically invariant mean.
\end{enumerate}
\end{mainthm}

\begin{mainthmtwo}
Let $X$ be a discrete metric space. Then the following are equivalent:
\begin{enumerate}
\item $\HQA^q(X, \mathcal V)= 0$ for all $q\geq 1$ and all modules $\mathcal V$ over $X$.
\item $[\J_Q^{0,1}]=0$ in $\HQA^1(X, \mathcal \ell^1_0(X))$.
\item $X$ has property $A$.
\end{enumerate}
\end{mainthmtwo}

\begin{mainthmthree}
Let $X$ be a discrete metric space. Then the following are equivalent:
\begin{enumerate}
\item $\HWA^q(X, \mathcal V)= 0$ for all $q\geq 1$ and all modules $\mathcal V$ over $X$.
\item $[\J_W^{0,1}]=0$ in $\HWA^1(X, \mathcal \ell^1_0(X))$.
\item $X$ has property $A$.
\end{enumerate}
\end{mainthmthree}

}

\skipit{

In fact the two theories will be defined in exactly the same way apart from the fact that the second invokes equivariance in every place where it makes sense to do so. The cohomology theory $H^m_B(G,\mathcal V)$ will be in some sense equivalent to the standard bounded cohomology theory once we have chosen appropriate modules. It is apparent that in fact it is sufficient for a particular cohomology group, indeed a particular class,  to be trivial in order to establish that the group is amenable. Classically this class lives in a space of means on the group, and vanishing of this class provides an invariant mean establishing amenability directly. In our theory vanishing of a specific class furnishes the group with a F\o lner sequence, which for a discrete group is equivalent to amenability. In the case of $H^1_A(X, \ell^1_0X)$,  vanishing of the class $\delta_z-\delta_y$ furnishes the space with a family of probability measures which exhibit property A. We note in passing that these functions are cochains in $E^{-1}_{ai}(X, \ell^1X)$.

}

\skipit{

\begin{section}{Homological characterisations of amenability}\label{exact}
The purpose of this section is to illuminate the relationship between the two following remarkable characterisations of amenability for a group. The definitions will follow the statements.

\begin{thm} (Ringrose and Johnson) A group $G$ is amenable if and only if $H^1_b(G,(\ell^{\infty}(G)/\C)^*)=0$
\end{thm}

\begin{thm} (Block and Weinberger) A group $G$ is amenable if and only if $H_0(G, \ell^\infty G)\not = 0$.
\end{thm}

It should be noted that both statements are part of a much larger picture. In the case of bounded cohomology vanishing of the first cohomology with the given coefficients is guaranteed by the triviality of a particular cocycle, and furthermore this ensures triviality of bounded cohomology with any coefficients in dimensions greater than or equal to $1$. In the case of Block and Weinberger's uniformly finite homology, vanishing of the zero dimensional homology group is guaranteed by the triviality of a fundamental class. It should also be noted that the definition they gave applies in the much wider context of an arbitrary metric space.

Recall the following definitions.
\begin{defn}
A \emph{mean} on a group $G$ is a positive linear functional $\mu$ on $\ell^\infty G$ such that $\norm{\mu}=1$. A group $G$ is \emph{amenable} if it admits a $G$-invariant mean.
\end{defn}

Recall that for a Banach space $V$ equipped with an isometric action of a group $G$, $C_b^m(G,V^*)$ denotes the $G$-module of equivariant bounded cochains $\phi:G^{m+1}\rightarrow V^*$. (Here bounded is defined by the Banach norm on the dual space $V^*$). This yields a cochain complex $(C_b^m(G,V^*), d)$ where $d$ denotes the natural differential induced by the homogeneous bar resolution. The cohomology of this complex is the bounded cohomology of the group with coefficients in $V^*$, denoted $H^*_b(G, V^*)$. For $V=\ell^\infty G/\C$ there is a particular class in dimension $1$which detects amenability which we will call the Johnson element. This is represented by the function 

\[
J(g_0, g_1)=\delta_{g_1}-\delta_{g_0},
\]

where $\delta_g$ denotes the Dirac delta function supported at $g$. Note that $J(g_0, g_1)$ lies in $\ell^1_0(G)$ which we view as a subspace of its double dual, $V^*$.

Dually we have the chain complex $(C_m^{\ell^1}(G,V), \partial)$, where $C_m^{\ell^1}(G,V)$ consists of equivariant functions $c:G^{m+1}\rightarrow V$ which are $\ell^1$ on the subspace $\{e\}\times G^m$. The boundary map is defined by 

\[
\partial c(g_0, \ldots, g_{m-1})=\sum\limits_{g\in G, i\in\{0,\ldots, m\} } (-1)^ic(g_0, \ldots, g_{i-1}, g, g_i, \ldots, g_{m-1}).
\]
 
 The homology of this complex is the $\ell^1$-homology of the group with coefficients in  $V$, denoted $H^{\ell^1}_*(G, V)$. Note that there is a map $H_*(G, V)\rightarrow H^{\ell^1}_*(G, V)$ given by the forgetful map. The fundamental class of Block and Weinberger in $H_0(G, \ell^\infty G)$ is represented by the cycle $c:G\rightarrow \ell^\infty G$ defined by $c(g)(h)=1$ for all $g,h\in G$. Applying the forgetful functor we obtain an element of $H^{\ell^1}_0(G, \ell^\infty G)$, and we will see that non-vanishing of this also characterises amenability. 
 
 We note that the pairing of $V^*$ with $V$, denoted $\langle-. -\rangle_V$ induces a pairing of $H^m_b(G, V^*)$ with $H_m^{\ell^1}(G, V)$ defined by 
 
 \[
 \langle[\phi], [c]\rangle = \sum\limits_{g_1,\ldots, g_m\in G}\langle\phi(e, g_1, \ldots, g_m), c(e, g_1, \ldots, g_m)\rangle_V.
 \]
 
It is clear that the pairing is defined at the level of cochains. To verify that it is well defined on classes one checks that the differential $d$ is the adjoint of the boundary map $\partial$.

The proof of the following result is a standard application of the snake lemma:

\begin{prop}\label{ses}
The short exact sequence of $G$-modules
$$0\to \C \xrightarrow{\iota} \ell^\infty G \xrightarrow{\pi} \ell^\infty G/\C \to 0.$$
induces a short exact sequence of chain complexes
$$0\to C_m^{\ell^1}(G,\C)\xrightarrow{\iota} C_m^{\ell^1}(G,\ell^\infty G) \xrightarrow{\pi} C_m^{\ell^1}(G,\ell^\infty G/\C)\to 0$$
and hence a long exact sequence of $\ell^1$-homology groups.

The short exact sequence of $G$-modules
$$0\to (\ell^\infty G/\C)^* \xrightarrow{\pi^*} \ell^\infty G^*  \xrightarrow{\iota^*} \C\to 0$$
induces a short exact sequence of cochain complexes
$$0\to C^m_{b}(G,(\ell^\infty G/\C)^*)\xrightarrow{\pi^*} C^m_{b}(G,\ell^\infty G^*) \xrightarrow{\iota^*} C^m_{b}(G,\C)\to 0$$
and hence a long exact sequence of bounded cohomology groups.
\qed
\end{prop}

Let $\mathbf 1$ denote the constant function $G\rightarrow \mathbb C$ which takes the value $1$ at every $g\in G$. This function represents classes in all of the following objects: $H^0_b(G, \C)$, $H_0(G, \C)$, $H^{\ell^1}_0(G, \C)$. Our point of view is that the Block-Weinberger fundamental class is $i[\mathbf 1]\in H_0(G, \ell^\infty G)$, while the Johnson cocycle is $d[\mathbf 1]\in H^1_b(G, (\ell^\infty G/\C)^*)$, where $d$ denotes the connecting map $H^0_b(G, \C)\rightarrow H^1_b(G, (\ell^\infty G/\C)^*)$. The first of these observations is elementary. For the second, note that $d[\mathbf 1]$ is obtained by lifting $\mathbf 1$ to the element $g\mapsto \delta_g$ in $C^0_b(G,(\ell^\infty G)^*)$ and taking the coboundary. This produces the Johnson cocycle $J(g_0,g_1)=\delta_{g_1}-\delta_{g_0}$.

By exploiting the connecting maps arising in Proposition \ref{ses} together with these observations we will obtain a new proof that $G$ is amenable if and only if the Johnson cocycle in bounded cohomology vanishes, and that this is equivalent to non-vanishing of the Block-Weinberger fundamental class. The first hint of the interaction is provided by the observation that dualising $H_0(G, \ell^\infty G)$ we obtain $H^0(G, \ell^\infty (G)^*)$ and that this is equal to $H^0_b(G, \ell^\infty (G)^*)$ since equivariance ensures that 0-cochains are bounded. The non-vanishing of $H^0_b(G, \ell^\infty (G)^*)$ is equivalent to amenability, since elements of  $H^0_{b}(G,\ell^\infty G^*)$ are maps $\phi:G\rightarrow \ell^\infty G^*$, which are $G$-equivariant and also, since they are cocycles, constant on $G$. Hence the value of a cocycle $\phi$ at any (and hence all) $g\in G$ is a $G$-invariant linear functional on $\ell^\infty G$. If $\phi$ is non-zero then taking its absolute value and normalising we obtain an invariant mean on the group. Conversely any invariant mean on the group is an invariant linear functional on $\ell^\infty G$ and hence gives a non-zero element of $H^0_{b}(G,\ell^\infty G^*)$.

\begin{thm}\label{beautiful}
Let $G$ be a countable discrete group. The following are equivalent:
\begin{enumerate}
\item \label{amen} $G$ is amenable.

\item \label{inv-mean} $\iota^*:H^0_{b}(G,\ell^\infty G^*) \to H^0_{b}(G,\C)$ is surjective.

\item \label{johnson} The Johnson class $d[\mathbf 1]$ vanishes in $H^1_{b}(G,(\ell^\infty G/\C)^*)$. 


\item \label{pairing} $\langle d[\mathbf 1],[c]\rangle =0$ for all $[c]$ in $H_1^{\ell^1}(G,\ell^\infty G/\C)$. (Hence for a non-amenable group, the non-triviality of $d[\mathbf 1]$ is detected by the pairing.)

\item \label{hom-last} $\iota[\mathbf 1] \in H_0^{\ell^1}(G,\ell^\infty G)$ is non-zero.


\item \label{B-W}The Block-Weinberger fundamental class $\iota[\mathbf 1] \in H_0(G,\ell^\infty G)$ is non-zero. 


\end{enumerate}
\end{thm}

\begin{proof}
(\ref{amen})\imp (\ref{inv-mean}) since $H^0_{b}(G,\C)=\C$, and for $\mu$ an invariant mean $i^*[\mu]=[\mathbf 1]$.

(\ref{inv-mean}) $\iff$ (\ref{johnson}): By exactness, surjectivity of $\iota^*$ is equivalent to vanishing of $d$, hence in particular this implies $d[\mathbf 1]=0$. The converse follows from the fact that $[\mathbf 1]$ generates $H^0_{b}(G,\C)$, so if $d[\mathbf 1]=0$ then $d=0$ and $\iota^*$ is surjective.

The implication 
(\ref{johnson})\imp (\ref{pairing}) is trivial.


(\ref{pairing}) $\implies$ (\ref{hom-last}): (\ref{pairing}) is equivalent to $\langle [\mathbf 1],\partial[c]\rangle =0$ for all $[c]$ in $H_1^{\ell^1}(G,\ell^\infty G/\C)$ by duality. We note that the space of 0-cycles in $C_0^{\ell^1}(G,\C)$ is $\C$, and noting that the pairing of the class $[\mathbf 1]$ in $H^0_{b}(G,\C)$ with the class $[\mathbf 1]$ in $H_0^{\ell^1}(G,\C)$ is $\langle [\mathbf 1],[\mathbf 1]\rangle =1$, we see that $[\mathbf 1]\in H_0^{\ell^1}(G,\C)$ is not a boundary. Thus $H_0^{\ell^1}(G,\C)=\C$ and the pairing with $H^0_{b}(G,\C)$ is faithful so $\langle [\mathbf 1],\partial[c]\rangle =0$ for all $[c]$ implies $\partial=0$. From this we deduce that $\iota$ is injective by exactness, hence we have (\ref{hom-last}): $\iota[\mathbf 1]$ is non-zero.



(\ref{hom-last})\imp (\ref{B-W}) since $\iota[\mathbf 1] \in H_0^{\ell^1}(G,\ell^\infty G)$ is the image of the corresponding element of $H_0(G,\ell^\infty G)$ under the forgetful map.

(\ref{B-W})\imp (\ref{amen}): We will use an argument due to Nowak. Let $\delta:C^0_{}(G,\ell^1(G))\to C^1_{}(G,\ell^1(G))$ denote the restriction of $d$. This is the pre-dual of $\partial$. First we note that  $\delta$ is not bounded below, since if it were then  $\partial=\delta^*$ would be surjective and $H_0(G, \ell^\infty G)$ would vanish giving $\iota[\mathbf 1]=0$, which is a contradiction.

The fact that $\delta$ is not bounded below is precisely the assertion that there is a Reiter sequence for the group and that therefore it is amenable. \note{Do we need to say what a Reiter sequence is?}

\end{proof}

As an example of this approach we give a proof of non-amenability for $F_2$ by constructing an explicit element $[c]\in H_1^{\ell^1}(G,\ell^\infty G/\C)$ for which $\langle d[\mathbf 1],[c]\rangle\not= 0$.

Let $\{a,b\}$ be a free basis for $F_2$, and let $\Gamma$ denote the Cayley graph of $F$ with respect to this generating set. $\Gamma$ is a tree and the action of $G$ on $\Gamma$ extends to the Gromov boundary. We choose a point $p$ in the Gromov boundary of $\Gamma$. For the sake of definiteness we set $p$ to be the endpoint of the ray $(a^n)$ where $n$ ranges over the positive integers, though this is not essential.

For a generator $g$ of $F_2$ (or its inverse) we set $c(e,g)(h)= 1$ if $(e, g)$ is the first edge on the geodesic from $e$ to $hp$ and set $c(e, g)(h)=0$ otherwise. Extending the definition by equivariance we obtain a function $c$ defined on the edges of $\Gamma$ with values in $\ell^\infty G$ and this represents an element $\overline c\in\ell^\infty G/\mathbb C$. 

Now consider $\partial c(e)=\sum\limits_{g\in\{a^{\pm 1}, b^{\pm 1}\}}c(g,e)-c(e,g)$.

For a given $h$ exactly one of the edges $(e,a), (e,b), (e, a^{-1}), (e, b^{-1})$ is the first edge on the geodesic $[e, hp]$, so the sum  $c(e,a)+ c(e,b)+ c(e, a^{-1})+c(e, b^{-1})$ is the constant function $\mathbf 1$ on $G$. 

On the other hand for a generator $g$, $c(g,e)(h)=1$ if and only if the edge $(g, e)$ is the first edge on the geodesic from $g$ to $hp$. We now consider the function $c(a,e)+c(b,e)+c(a^{-1}, e)+c(b^{-1},e)$. For a given $h\in G$ there is a unique point in the set $\{a,b, a^{-1}, b^{-1}\}$ which lies on the geodesic from $e$ to $hp$, and this is the only one for which the corresponding term of the sum takes the value $0$, so the sum $c(a,e)+c(b,e)+c(a^{-1}, e)+c(b^{-1},e)$ is the constant function $\mathbf 3$.

Hence $\partial c(e)=\mathbf 3-\mathbf 1=\mathbf 2$. Now by equivariance $\partial c(k)=\mathbf 2$ for all $k$, hence $\partial \overline c$ vanishes in $\ell^\infty (G)/\mathbb C$, so $\overline c$ is  a cycle and therefore represents an element $[\overline c]\in H_1^{\ell^1}(G,\ell^\infty G/\C)$.

We now compute the pairing $\langle d[\mathbf 1], [\overline c]\rangle$. 

\[
\langle d[\mathbf 1], [\overline c]\rangle=\langle [\mathbf 1], \partial [\overline c]\rangle=\langle [\mathbf 1], [\partial  c]\rangle = \langle [\mathbf 1], [\mathbf 2]\rangle=2.
\]


Hence $F_2$ is not amenable.

We conclude this section by noting that amenability is also equivalent to vanishing of the Johnson class as an element of  $H^1(G,(\ell^\infty G/\C)^*)$. To see this, replace the pairing of    $H^1_{b}(G,(\ell^\infty G/\C)^*)$  and $H_1^{\ell^1}(G,\ell^\infty G/\C)$ in the proof of Theorem \ref{beautiful} with the standard pairing of $H^1(G,(\ell^\infty G/\C)^*)$  and $H_1(G,\ell^\infty G/\C)$, hence deducing that vanishing of the Johnson element in $H^1(G,(\ell^\infty G/\C)^*)$ implies non-vanishing of the Block-Weinberger fundamental class:

\begin{thm}\label{unboundedbeauty}
Let $G$ be a countable discrete group. The following are equivalent:
\begin{enumerate}
\item $G$ is amenable.

\item  $\mathbf 1$ lies in the image of $i^*:H^0_{}(G,\ell^\infty G^*) \to H^0_{}(G,\C)$.

\item The Johnson class $d[\mathbf 1]$ vanishes in $H^1_{}(G,(\ell^\infty G/\C)^*)$. 

\end{enumerate}

\end{thm}

\end{section}

}

\begin{section}{The cochain complex}

Let $X$ be a metric space, $G$ be a group acting by isometries on $X$ and $\mathcal V=(V, \|\cdot\|_V, \text{Supp})$ be a $G$-equivariant $X$ module. Associated to this data we will construct a bicomplex $\Epq\XV$. This bicomplex also depends on the group $G$, however for concision we will generally omit $G$ from our notation.

This bicomplex is in some sense too small to detect property A, and we will construct two `completions' of the bicomplex to rectify this, yielding two cohomology theories $\HQ^p\XV$ and $\HW^p\XV$.

There are two principal cases of interest. When $G$ is trivial, we obtain non-equivariant cohomologies detecting property A. When $X=G$ is a group acting on itself by left multiplication and equipped with a proper left invariant metric, the cohomologies detect amenability for $G$.

For $\mathbf x\in X^{p+1}, \mathbf y \in X^{q+1}$, we make the standard convention that coordinates of $\mathbf x, \mathbf y$ are written $x_0,\dots,x_p$ and $y_0,\dots,y_q$.

For a positive real number $R$ let $\Delta_R^{p+1}$ denote the set $\{\mathbf x\in X^{p+1} \mid d(x_i, x_j)\leq R, \, \forall i,j\}$, and let $\Delta_R^{p+1,q+1}$ denote the set
$$\bigl\{\bxy\in X^{p+1}\times X^{q+1} \mid d(u,v)\leq R, \, \forall u,v\in\{x_0,\dots x_p,y_0,\dots,y_q\}\bigr\}.$$
Identifying $X^{p+1}\times X^{q+1}$ with $X^{p+q+2}$ in the obvious way, $\Delta_R^{p+1,q+1}$ can be identified with $\Delta_R^{p+q+2}$.

Given a function $\phi:X^{p+1}\times X^{q+1} \rightarrow V$   we set 
\[\|\phi\|_R=\sup_{\mathbf x \in \Delta_R^{p+1}, \mathbf y\in X^{q+1}} \|\phi(\mathbf x, \mathbf y)\|_V.\]




We say that a function $\phi$ is of controlled supports if for every $R>0$ there exists $S>0$ such that whenever $(\mathbf x,\mathbf y)\in \Delta^{p+1,q+1}_R$ then $\text{Supp}(\phi\bxy)$ is contained in $B_{S}(x_i)$ and $B_{S}(y_j)$ for all $i,j$.

The action of $G$ on $X$ extends to give diagonal actions on $X^{p+1}$ and $X^{q+1}$, and a function $\phi$ as above is said to be equivariant if for every $g\in G$

\[
g(\phi(\mathbf x,\mathbf y))=\phi(g\mathbf x, g\mathbf y).
\]

We define
\[
\Epq\XV=\{\phi:X^{p+1}\times X^{q+1} \rightarrow V \mid \text{ $\phi$ is equivariant, of controlled supports and } \|\phi\|_R< \infty\;  \forall R> 0 \}.
\]

We equip the space $\Epq\XV$ with the topology arising from the semi-norms $\|\cdot\|_R$. It seems natural to allow $R$ to range over all positive values, however we note that the topology this gives rise to is the same as the topology arising from the countable family of seminorms $\|\cdot\|_R$ for $R\in\naturals$. 


The usual boundary map $\partial: X^{m+1}\mapsto X^m$  induces a pair of anti-commuting coboundary maps which yields a bicomplex 

\[
\begin{CD}
@.@A{D}AA  @A{D}AA @A{D}AA\\
\qquad\qquad @.\E^{2,0} @>{d}>> \E^{2,1} @>{d}>> \E^{2,2} @>{d}>>\\
@A{p}AA @A{D}AA  @A{D}AA @A{D}AA\\
@.\E^{1,0} @>{d}>> \E^{1,1} @>{d}>> \E^{1,2} @>{d}>>\\
@.@A{D}AA  @A{D}AA @A{D}AA\\
@.\E^{0,0} @>{d}>> \E^{0,1} @>{d}>> \E^{0,2} @>{d}>>\vspace{1.5ex}\\
@.@. @>>{q}>
\end{CD}
\]
Specifically, $D: \E^{p,q}\rightarrow \E^{p+1,q}$ is given by
\[
 D\phi \left((x_0, \dots, x_{p+1}), \mathbf{y}\right) = \sum_{i=0}^{p+1}(-1)^{i}
\phi((x_0, \ldots, \widehat{x_i}, \ldots, x_{p+1}), \mathbf{y})
\]
while $d: \E^{p,q}\rightarrow \E^{p,q+1}$ is 
\[
d\phi(\mathbf{x}, (y_0, \dots, y_{q+1})) = \sum_{i=0}^{q+1} (-1)^{i+p} \phi(\mathbf{x}, (y_0, \dots, \widehat{y_i}, \dots, y_{q+1})).
\]

In Proposition \ref{acyclicrows} we will show that the rows of our bicomplex are acyclic, and for this reason it makes sense to consider an augmentation of the rows making them exact at $q=0$. We note that the definition of $\Epq$ and the maps $D:\E^{p,q}\to \E^{p+1,q}, d:\E^{p,q}\to \E^{p,q+1}$ make sense not just for positive $p,q$ but also when one of $p$ or $q$ is $-1$. We will be interested in $\E^{p,-1}\XV$ which we will identify as the augmentation of row $p$. The augmentation map is the differential $d\E^{p,-1}\to \E^{p,0}$.

Elements of $\E^{p,-1}\XV$ are maps $\phi:X^{p+1}\times X^0\to V$; for convenience of notation, we will suppress the $X^0$ factor, and write $\phi(\mathbf x)$ for $\phi(\mathbf x,())$. We note that the differential $d:\E^{p,-1}\to \E^{p,0}$  is defined by $d\phi(\mathbf x,(y))=\phi(\mathbf x,(\widehat y))$. Suppressing the empty vector we see that $d\phi(\mathbf x,(y))=\phi(\mathbf x)$, i.e.\  $d$ is the inclusion of $\E^{p,-1}\XV$ into $\E^{p,0}\XV$ as functions which are constant in the $y$ variable.

\begin{lemma}
The maps $D$ and $d$ are well-defined, continuous, anti-commuting differentials. 
\end{lemma}
\begin{proof}
The fact that $D$ and $d$ are anti-commuting differentials on the larger space of \emph{all} equivariant functions from $X^{p+1}\times X^{q+1}$ to $V$ is standard. We must show that $D,d$ preserve finiteness of the semi-norms, and controlled supports. We note that $\|D\phi\|_R \leq (p+2) \|\phi\|_R$ by the triangle inequality, and a corresponding estimate holds for $\|d\phi\|_R$. Hence $D,d$ are continuous, and the semi-norms are finite as required. 

For $\phi$ of controlled supports we now show that $D\phi$ is of controlled supports. Given $R>0$, take $\bxy \in \Delta^{p+2,q+1}_R$. Since $\phi$ is of controlled supports, there exists $S$ such that $\supp(\phi((x_0, \ldots, \widehat{x_i}, \ldots, x_{p+1}), \mathbf{y}))$ is contained in $B_{S}(x_{i'})$ and $B_S(y_j)$ for all $i'\neq i$, and for all $j$. Since for any $i'\neq i$ we have $d(x_i,x_{i'})\leq R$ we deduce that  $\supp(\phi((x_0, \ldots, \widehat{x_i}, \ldots, x_{p+1}), \mathbf{y}))$ lies in $B_{S+R}(x_{i'})$ for all $i'$. By the axioms for $\supp$ the support of $D\phi$ is contained in $B_{S+R}(x_{i'})$ and $B_S(y_j)$ for all $i'$ and all $j$, since this holds for the summands.

The argument for $d\phi$ is identical, exchanging the roles of $\mathbf x,\mathbf y$.
\end{proof}
\bigskip

Let $\He^*\XV$ denote the cohomology of the totalisation of the bicomplex $\E^{p,q}, p,q\geq 0$, with the differentials $D,d$.

\begin{remark}
If $X$ is equipped with two coarsely equivalent $G$-invariant metrics $d,d'$ then for any module over $X$ the controlled support conditions arising from these metrics are the same. Moreover the family of semi-norms is equivalent in the sense that for each $R$ there is an $S$ such that $\|\cdot\|_{R,d}\leq \|\cdot\|_{S,d'}$ and for each $R$ there is an $S$ such that$\|\cdot\|_{R,d}\leq \|\cdot\|_{S,d'}$. Hence the bicomplexes and the cohomology we obtain from each metric are identical. This applies in particular if $X=G$ is a countable group and the two metrics are both left-invariant proper metrics on $G$.
\end{remark}

We will now demonstrate exactness of the rows. This allows the cohomology of the totalisation to be computed in terms of the left-hand column.

\begin{prop}\label{acyclicrows}
For each $p$ the augmented row $(\E^{p,*},d)$, $p\geq -1$ is exact.

Specifically, for all $p\geq 0$ there is a continuous splitting $s:\E^{p,q}\to \E^{p,q-1}$ given by
\[
s\phi((x_0,\dots,x_p),(y_0,\dots,y_{q-1}))= (-1)^p\phi((x_0,\dots,x_p),(x_0,y_0,\dots,y_{q-1})).
\]
We have $(ds+sd)\phi=\phi$ for $\phi\in \Epq$ with $p\geq 0$, and $sd\phi=\phi$ for $\phi$ in $\E^{p,-1}$
\end{prop}

\begin{proof}
The fact that $s$ defines a splitting on the larger space of all equivariant functions from $X^{p+1}\times X^{q+1}$ to $V$ is standard homological algebra. We must verify that $s$ is continuous, from which it immediately follows that $s\phi$ has bounded $R$-norms, and that if $\phi$ is of controlled supports then so is $s\phi$.




Continuity is straightforward. For each $R\geq 0$ we have $\|s\phi\|_R\leq \|\phi\|_R$; this is immediate from the observation that if $(x_0,\dots,x_p)$ in $\Delta^{p+1}_R$ then
$$\|\phi((x_0,\dots,x_p),(x_0,y_0,\dots,y_{q-1}))\|_V\leq \|\phi\|_R.$$

It remains to verify that $s\phi$ is of controlled supports. Given $R>0$, since $\phi$ is of controlled supports we know there exists $S$ such that if $\bxy\in \Delta^{p++1,q+1}_R$ then $\supp(\phi\bxy)$ is contained in $B_{S}(x_i)$ and $B_{S}(y_j)$ for all $i,j$. If $((x_0,\dots,x_p),(y_0,\dots,y_{q-1}))\in \Delta^{p+1,q}_R$ then we have $((x_0,\dots,x_p),(x_0,y_0,\dots,y_{q-1}))\in \Delta^{p+1,q+1}_R$, hence $\supp(s\phi((x_0,\dots,x_p),(y_0,\dots,y_{q-1})))$ is also contained in $B_{S}(x_i)$ and $B_{S}(y_j)$ for all $i,j$.


This completes the proof.
\end{proof}

We remark that the corresponding statement is false for the vertical differential $D$, since for $\phi\in \E^{p,q}\XV$, the function $((x_0,\dots,x_{p-1}),(y_0,\dots,y_q)) \mapsto \phi((y_0,x_0,\dots,x_{p-1}),(y_0,\dots,y_{q})$ is only guaranteed to be bounded on sets of the form $\bigl\{((x_0,\dots,x_{p-1}),(y_0,\dots,y_q)) \mid d(u,v)\leq R\text{ for all }u,v \in \{x_0,\dots,x_{p-1},y_0\}\bigr\}$, and not on $\Delta^{p}_R\times X^{q+1}$. Hence the vertical `splitting' would not map $\E^{p,q}$ into $\E^{p-1,q}$.

\begin{corollary}
\label{acorollary}
The cohomology $\He^*\XV$ is isomorphic to the cohomology of the cochain complex $(\E^{*,-1},D)$.
\end{corollary}

\begin{proof}
This follows from the exactness of the augmented rows of the bicomplex. Each cocycle in $\E^{p,q}\XV$ is cohomologous to a cocycle in $\E^{p+q,0}\XV$, whence $\He^*\XV$ is isomorphic to the cohomology of the complex $\ker(d:\E^{p,0}\to \E^{p,1})$ with the differential $D$. The augmentation map $d:\E^{p,-1}\to \E^{p,0}$ yields an isomorphism from $(\E^{p,-1},D)$ to the kernel $\ker(d:\E^{p,0}\to \E^{p,1})$, and as $D,d$ anti-commute, the differential $D$ on the kernels is identified with the differential $-D$ on $\E^{p,-1}$. We note however that the change of sign does not affect the cohomology, so $\He^*\XV$ is isomorphic to the cohomology of $(\E^{*,-1},D)$ as claimed.
\end{proof}

The cohomology $\He\XV$ is not sufficiently subtle to detect property A. In the following section we will introduce two completion procedures which we will apply to the bicomplex $\E$, obtaining more refined cohomologies.

\end{section}

\section{Generalised completions}

\def\something{pre-Fr\'echet }

Let $\E$ be a vector space equipped with a countable family of seminorms $\|\cdot\|_i$ which separates points. We will call such a space a \something space. We have in mind that $\E=\Epq\XV$, for some $p,q,X,G$ and $\mathcal V$.

If $\E$ is not complete then one constructs the classical completion of $\E$ as follows. Let $E_{\mathrm{cs}}$ denote the space of Cauchy sequences in $\E$ (i.e.\ sequences which are Cauchy with respect to each semi-norm), and let $E_0$ denote the space of sequences in $\E$ which converge to 0. Then the completion of $\E$ is precisely the quotient space $E_{\mathrm{cs}}/E_0$. As the topology of $\E$ is given by a countable family of seminorms, this completion is a Fr\'echet space.

In this section we will define two generalised completions which are somewhat larger than the classical one, and we will demonstrate various properties of the completions, and relations between the two.

The first completion is motivated by the classical case.

\begin{defn}
The \emph{quotient completion} of $\E$, denoted $\EQ$ is the quotient space $E/E_0$ where $E$ denotes the space of bounded sequences in $\E$ and $E_0$ denotes the space of sequences in $\E$ which converge to 0.
\end{defn}

The second completion comes from functional analysis.

\begin{defn}
The \emph{weak-* completion} of $\E$, denoted $\EW$ is the double dual of $\E$.
\end{defn}

The space $\E$ is not assumed to be complete, however this does not matter as the dual of $\E$ is the same as the dual of it's completion $\overline{\E}=E_{\mathrm{cs}}/E_0$. Indeed we note that both $\EQ$ and $\EW$ contain $\overline{\E}$, and indeed $\EQ,\EW$ are respectively isomorphic to the quotient and weak-* completions of $\overline{\E}$.

Since the space $\E$ need not be a normed space, we recall some basic theory of duals of Fr\'echet spaces. For simplicity we assume that the seminorms on $\E$ are monotonic, i.e.\ $\|\cdot \|_i\leq \|\cdot \|_j$ for $i<j$, this being easy to arrange.

For $\alpha\in\E^*$, we can define $\|\alpha\|^i=\sup\{|\la\alpha,\phi\ra| \mid \|\phi\|_i\leq 1\}$. We note that $\|\alpha\|^i$ takes values in $[0,\infty]$, and $\|\cdot \|^i\geq \|\cdot \|^j$ for $i<j$. The condition that $\alpha$ is continuous is the condition that $\|\alpha\|_i$ is finite for some $i$. For any sequence $r_1,r_2,\dots$ the set $\{\alpha \in \E^* \mid \|\alpha\|^i<r_i \text{ for some } i\}$ is a neighbourhood of $0$, and every neighbourhood of 0 contains such a set. Hence these sets determine the topology on $\E^*$.


Having equipped $\E^*$ with this topology, we can then form the space $\E^{**}$ of continuous linear functionals on $\E^*$. A linear functional $\eta$ on $\E^*$ is continuous if for all $i$, we have $\|\eta\|_i=\sup\{|\la\eta,\alpha\ra| \mid \|\alpha\|^i\leq 1\}<\infty$.

The space $\EW=\E^{**}$ will be equipped with the weak-* topology. It follows by the Banach-Alaoglu theorem that all bounded subsets of $\EW$ are relatively compact. In the language of Bourbaki, if $A\subseteq \EW$ is bounded, i.e.\ there exists a sequence $r_i$ such that $\|\eta\|_i\leq r_i$ for all $i$, then $A$ is contained in the polar of $\{\alpha\in \E^* \mid \exists i,\,\|\alpha\|^i\leq 1/r_i\}$, which is compact.

\begin{remark}
From an abstract perspective, the weak-* completion is a natural way to enlarge $\E$. On the other hand, from the point of view of explicitly constructing elements of the space, the quotient completion is more tractable.
\end{remark}

\begin{defn}
We say that a short exact sequence $0\to\E\xrightarrow i\E'\xrightarrow \pi\E''\to 0$ of locally convex topological vector spaces is \emph{topologically exact} if the maps $i,\pi$ are open.
\end{defn}
Note that if the spaces are complete then the requirement that $i,\pi$ are open is automatic by the open mapping theorem.

\begin{prop}
\label{exact-completion}
Let $\E,\E'$ be \something spaces. Then a continuous map $T:\E\to\E'$ extends to give maps $T_Q:\EQ\to \EQ'$ and $T_W:\EW\to \EW'$. Moreover this process respects compositions, and takes short topologically exact sequences to short exact sequences.
\end{prop}

\begin{proof}
For the quotient completion, continuity of the map $T:\E\to \E'$ guarantees that applying $T$ to each term of a bounded sequence $\phi_n$ in $\E$ we obtain a bounded sequence $T\phi_n$ in $\E'$. If $\phi_n\to 0$ then $T\phi_n\to 0$ by continuity, hence we obtain a map $T_Q:\EQ\to \EQ'$. It is clear that this respects compositions.

Now suppose $0\to\E\xrightarrow i\E'\xrightarrow \pi\E''\to 0$ is a short exact sequence.
If $i_Q$ vanishes on a coset $[(\phi_n)]\in \EQ$ then $i\phi_n\to 0$. Since $i$ is open and injective, $i\phi_n\to 0$ implies $\phi_n\to 0$. Hence $[(\phi_n)]=0$ and we have shown that $i_Q$ is injective.

If $[(\phi'_n)]\in \EQ'$ with $\pi[(\phi'_n)]=0$ then $\pi\phi_n'\to 0$. The map $\pi$ induces a map $\E'/i\E\to \E''$ which is an isomorphism of \something spaces, hence the class of $\phi_n'$ tends to 0 in the quotient $\E'/i\E$. That is, there exists a sequence $\psi'_n$ in the $i\E$ such that $\phi_n'-\psi_n'\to 0$. We have $\psi'_n=i\psi_n$ for some $\psi_n\in \E$, and $[(\phi'_n)]=[(\psi'_n)]=[(i\psi_n)]$ in $\EQ'$, hence we deduce that $[(\phi'_n)]$ is in the image of $i_Q$.

Finally, for surjectivity of $\pi_Q$ we note that if $[(\phi''_n)]\in \EQ''$ then there exists $\phi'_n$ such that $\phi''_n=\pi\phi'_n$. By openness of $\pi$, the sequence $\phi'_n$ can be chosen to be bounded.

In the case of the weak-* completion, the maps $T_W,i_W,\pi_W$ are simply the double duals of $T,i,\pi$. The fact that this respects composition is then standard. The hypothesis that $i,\pi$ are open ensures that the corresponding sequence of classical completions is exact, whence exactness of the double duals is standard functional analysis.
\end{proof}

We now give a connection between the two completions.

\begin{prop}
\label{e_omega}
Let $\omega$ be a non-principal ultrafilter on $\naturals$. Then for any \something space $\E$ there is a linear map $e_\omega:\EQ\to \EW$, defined by $\la e_\omega(\phi),\alpha \ra=\lim\limits_{\omega}\la\alpha,\phi_n\ra$. Moreover for $T:\E\to \E'$ we have $e_\omega \circ T_Q=T_W\circ e_\omega$. If $I_Q,I_W$ denote the inclusions of $\E$ in $\EQ,\EW$ then $I_W=e_\omega \circ I_Q$ for all $\omega$.
\end{prop}

\begin{proof}
We view an element $\phi$ of $E$ as a map $\phi:\naturals \to \E\subseteq \EW$ with bounded range. Hence the closure of the range is compact in the weak-* topology on $\EW$ by the Banach-Alaoglu theorem. By the universal property of the Stone-\v Cech compactification it follows that $\phi$ extends to a map $\overline{\phi}:\beta\naturals \to \EW$, which is continuous with respect to the weak-* topology on $\EW$. We define $e_\omega(\phi)=\overline{\phi}(\omega)$.

Continuity of $\overline{\phi}$ guarantees that for each $\alpha\in \E^*$, we have $\la \overline{\phi}(\cdot),\alpha\ra$ a continuous function on $\beta\naturals$. This is the extension to $\beta\naturals$ of the bounded function $n\mapsto \la\alpha,\phi_n\ra$, hence evaluating at $\omega$ we have
$$\la e_\omega(\phi),\alpha \ra=\la \overline{\phi}(\omega),\alpha\ra=\lim\limits_{\omega}\la\alpha,\phi_n\ra.$$

The fact that $e_\omega \circ T_Q=T_W\circ e_\omega$ is now easily verified as
$$\la e_\omega(T_Q\phi),\alpha\ra=\lim\limits_{\omega}\la\alpha,T\phi_n\ra=\lim\limits_{\omega}\la T^*\alpha,\phi_n\ra=\la e_\omega(\phi),T^*\alpha\ra=\la T_We_\omega(\phi),\alpha\ra$$
for all $\alpha \in \E^*$. The final assertion is simply the observation that the extension to $\beta\naturals$ of a constant sequence is again constant.
\end{proof}

\medskip

We are now in a position to define our cohomology theories.

For $p\geq 0,q\geq -1$, let $\EQ^{p,q}\XV$ denote the quotient completion of $\E^{p,q}\XV$, and let $\EW^{p,q}\XV$ denote the weak-* completion of $\E^{p,q}\XV$. As $(D,d)$ are continuous anti-commuting differentials, the extensions of these to the completions (which we will also denote by $D,d$) are again anti-commuting differentials, hence taking $p,q\geq 0$ we have bicomplexes $(\EQ^{p,q}\XV,(D,d))$ and $(\EW^{p,q}\XV,(D,d))$.

Let $\HQ^*\XV$ denote the cohomology of the totalisation of the bicomplex $\EQ^{p,q}\XV, p,q\geq 0$, and let $\HW^*\XV$ denote the cohomology of the totalisation of the bicomplex $\EW^{p,q\XV}, p,q\geq 0$.

Since the splitting $s$ is continuous it also extends to the completions and we deduce that the augmented rows of the completed bicomplexes are exact. This gives rise to the following.

\begin{corollary}
\label{anothercorollary}
The cohomologies $\HQ^*\XV,\HW^*\XV$ are isomorphic respectively to the cohomologies of the cochain complexes $(\EQ^{*,-1},D)$,$(\EW^{*,-1},D)$.
\end{corollary}

The argument is identical to Corollary \ref{acorollary}

We note that the extension of $s$ to the completions ensures taking the kernel of $d:\E^{p,0}\to\E^{p,1}$ and then completing (in either way), yields the same result as first completing and then taking the kernel; one obtains the completion of $\E^{p,-1}$. The corresponding statement for $D$ would be false. The kernel of $D:\EQ^{0,q}\to\EQ^{1,q}$ will typically be much larger than the completion of the kernel of $D:\E^{0,q}\to\E^{1,q}$, and similarly for $\EW$. These vertical kernels are of great interest and we will study them in Section \ref{ai-section}.

We now make a connection between the three cohomology theories $\He,\HQ,\HW$.

\begin{thm}\label{IQIW}
The inclusions of $\Epq\XV$ in $\EQ^{p,q}\XV$ and $\EW^{p,q}\XV$ and the map $e_\omega:\EQ^{p,q}\XV\to \EW^{p,q}\XV$ induce maps at the level of cohomology:
$$\begin{matrix}
\vspace{1.5ex} \He^*\XV & \xrightarrow{I_Q} & \HQ^*\XV\\
\vspace{1.5ex} & \searrow^{I_W} & \downarrow {\scriptstyle{e_\omega}}\\
& & \HW^*\XV
\end{matrix}$$
The above diagram commutes. 
Moreover the kernels $\ker I_Q$ and $\ker I_W$ are equal, that is, a cocycle in $\Epq\XV$ is a coboundary in $\EQ^{p,q}\XV$ if and only if it is a coboundary in $\EW^{p,q}\XV$.

\end{thm}

\begin{proof}
The existence of the maps at the level of cohomology follows from the fact that $D,d$ commute with the inclusion maps and $e_\omega$. The diagram commutes at the level of cochains by Proposition \ref{e_omega}. It is then immediate that $\ker I_Q\subseteq \ker I_W$. It remains to prove that if $\phi$ is a cocycle in $\E^{p,q}\XV$ with $I_W\phi$ a coboundary, then $I_Q\phi$ is also a coboundary.

By exactness of the rows, every cocycle in $\E^{p,q}\XV$ is cohomologous to an element of $\E^{p+q,0}\XV$, hence without loss of generality we may assume that $q=0$. Moreover any cocycle in $\E^{p,0}\XV$ is $d\phi$ for some $\phi$ in $\E^{p,-1}\XV$, and the images of $d\phi$ under $I_Q,I_W$ will be coboundaries if and only if $I_Q\phi,I_W\phi$ are a coboundaries in the completions of the complex $(\E^{p,-1}\XV,D)$.

Suppose that $I_W\phi$ is a coboundary, that is viewing $\phi$ as an element of the double dual $\EW^{p,-1}$, there exists $\psi$ in $\EW^{p-1,-1}$ such that $D\psi=\phi$.  We now appeal to density of $\E^{p-1,-1}$ in $\EW^{p-1,-1}$ to deduce that there is a net $\theta_\lambda$ in $\E^{p-1,-1}$ converging to $\psi$ in the weak-* topology. By continuity of $D$ we have that $D\theta_\lambda \to D\psi=\phi$. As $D\theta_\lambda$ and $\phi$ lie in $\E^{p,-1}$, we have that this convergence is in the weak topology on $\E^{p,-1}$. On any locally convex topological vector space, a convex set is closed in the locally convex topology if and only if it is closed in the associated weak topology. Hence (as the locally convex topology of $\E^{p,-1}$ is metrizable) there is a sequence $\theta_n$ of convex combinations of the net $\theta_\lambda$ such that $D\theta_n$ converges to $\phi$ in the $R$-semi-norm topology on $\E^{p,-1}$. Thus $D[\theta_n]=I_Q\phi$ in $\EQ^{p,-1}$.

Hence $I_Q\phi$ is also a coboundary, as required.

\end{proof}

\section{Morphisms, change of coefficients and the long exact sequence in cohomology}
 Now we consider the effect on cohomology of varying the coefficient module. Let $X$ be a metric space, $G$ be a group acting by isometries on $X$ and let $\mathcal U=(U, |\cdot|_U, \supp_U)$ and $\mathcal V=(V, |\cdot|_V, \supp_V)$ be $G$-equivariant $X$-modules.

\begin{defn}\label{morphisms}
A $G$-equivariant \emph{$X$-morphism} from $\mathcal U$ to $\mathcal V$ is an equivariant bounded linear map $\Psi:U\rightarrow V$ for which there exists $S\geq 0$ such that for all $u\in U$, $\supp_V(\Psi(u))\subseteq B_S(\supp_U(u))$. When the group action is clear from the context, in particular when $G$ is trivial, we will simply refer to this as an $X$-morphism.

An $X$-morphism $\Psi$  is said to be a \emph{monomorphism} if it is injective and if there exists $T\geq 0$ such that for all $u\in U$, $\supp_U(u)\subseteq B_T(\supp_V(\Psi(u)))$.

An $X$-morphism $\Psi$  is said to be an \emph{epimorphism} if it is surjective and there exists $M\geq 0$ such that for all $R\geq 0$ there exists $S\geq 0$ such that for all $v\in V$ if $\supp_V(v)\subseteq B_R(x)$ then there exists $u\in \Psi^{-1}(v)$ such that $\|u\|_U\leq M\|v\|_V$ and $\supp_U(u)\subseteq B_S(x)$.

An $X$-morphism $\Psi$  is said to be an \emph{isomorphism} if it is both an epimorphism and a monomorphism.  

 \end{defn}

 Note that a surjective monomorphism is automatically an epimorphism and therefore an isomorphism.

 \note{Probably there is an injective epi that is not an isomorphism}
 
In this section we will use the usual convention that the term morphism refers to an $X$-morphism when both the space $X$ and the group $G$ are clear from the context.

We note that the concept of a monomorphism is constructed to ensure that $U$ may be viewed in some sense as a sub-module of $V$, while the concept of an epimorphism is designed to provide controlled splittings, providing, respectively, some notion of injectivity and surjectivity for supports.

Given a space $X$, and a group $G$ acting by isometries on $X$,   a short exact sequence of $X$-modules is a short exact sequence of Banach spaces 

\[
0\rightarrow U\xrightarrow{\iota} V\xrightarrow{\pi} W\rightarrow 0
\]

\noindent each with the structure of a $G$-equivariant $X$-module and where $\iota$ is a monomorphism of $X$-modules and $\pi$ is an epimorphism.

\begin{lemma} An $X$-morphism $\Psi:\mathcal U\rightarrow \mathcal V$ induces a continuous linear map $\Psi_*:\E^{p,q}(X, \mathcal U)\rightarrow \E^{p,q}(X, \mathcal V)$ commuting with both differentials. This extends to give maps on both completed bicomplexes.

A short exact sequence of $X$-modules induces a short exact sequence of bicomplexes for $\E,\EQ$ and $\EW$. Hence, by the snake lemma, we obtain long exact sequences in cohomology for $\He^{*}(X, -)$,$\HQ^{*}(X, -)$ and $\HW^{*}(X, -)$.
\end{lemma}

\begin{proof}
Given an element $\phi\in \E^{p,q}(X,\mathcal U)$,  we need to check that $\Psi_*(\phi)=\Psi\circ\phi$ lies in $\E^{p,q}(X,\mathcal V)$. Equivariance of $\Psi_*(\phi)$ follows from equivariance of $\Psi$ and $\phi$.

Note that for any $R\geq 0$, the $R$-norm $\|\Psi\circ\phi\|_R$ in $V$ is at most $\|\Psi\|$ times the $R$-norm of $\phi$ in $U$, hence the map is continuous. 

As $\phi$ has controlled supports in $\mathcal U$, for any $R>0$ there exists $S>0$ such that if $\bxy\in \Delta_R$ then $\supp_U(\phi(\mathbf x, \mathbf y))\subseteq B_S(x_i),B_S(y_j)$ for all $i,j$. It follows that $\supp_V(\Psi\circ\phi(\mathbf x, \mathbf y))\subseteq B_{S+S'}(x_i)$ for all $i$, where $S'$ is the constant given in the definition of a morphism. Hence $\psi\circ\phi$ is an element of $\E^{p,q}(X,\mathcal V)$. Combining this with continuity we deduce that $\Psi\circ\phi$ lies in $\E^{p,q}(X,\mathcal V)$.

The fact that $\Psi$ commutes with the differentials is immediate from linearity of $\Psi$ and the definitions of $D,d$. As the maps $\Psi_*$ is continuous, it induces maps on both completions.

Now suppose we are given a short exact sequence of $X$-modules
\[
0\rightarrow U\xrightarrow{\iota} V\xrightarrow{\pi} W\rightarrow 0.
\]

We will show that the sequence
\[
0\rightarrow E^{p,q}(X,\mathcal U)\xrightarrow{\iota_*} E^{p,q}(X,\mathcal V)\xrightarrow{\pi_*} E^{p,q}(X,\mathcal W)\rightarrow 0
\]
is topologically exact, i.e.\ it is exact and the maps are open.

Injectivity and openness of $\iota_*$ follows directly from the corresponding properties of $\iota$; as $\iota$ has closed range, it is open by the open mapping theorem.

Exactness at the middle term follows from the observation that if $\pi_*(\phi)=0$ then $\phi=\iota\circ \phi'$ for some function $\phi':X^{p+1}\times X^{q+1}\rightarrow U$, where $\phi'$ is uniquely defined by injectivity of $\iota$. We need to verify that $\phi'$ is an element of $\E^{p,q}(X,\mathcal U)$. Openness of $\iota$ yields the required norm estimates, whereas the support condition is satisfied because $\iota$ is a monomorphism, hence $\supp_U(\phi')\subseteq B_T(\supp_V(\iota\circ\phi'))= B_T(\supp_V(\phi))$ for some $T\geq 0$.

Surjectivity of $\pi_*$ follows from the definition of an epimorphism: Given $\phi: X^{p+1}\times X^{q+1}\rightarrow W$ which is an element of $\E^{p,q}(X, \mathcal W)$, for each $R>0$ there exists $S>0$ such that $\bxy\in\Delta_R\leq R$ implies that $\supp_W(\phi(\mathbf x, \mathbf y)\subseteq B_{S}(x_i),B_{S}(y_j)$ for all $i,j$. Since $\pi$ is an epimorphism, there exists $M,T>0$ such that for each $(\mathbf x, \mathbf y)$ there exists an element of $V$, which we denote $\phi'(\mathbf x, \mathbf y)$ such that $\|\phi'(\mathbf x, \mathbf y)\|_V\leq M\|\phi(\mathbf x, \mathbf y)\|_W$ and $\supp_V(\phi'(\mathbf x, \mathbf y))\subseteq B_{T}(x_i),B_{T}(y_j)$ for each $i,j$, so $\phi'$ is  of controlled supports and has finite $R$-norms as required. These estimates for the $R$-norms also ensure that $\pi_*$ is open.

Hence by Proposition \ref{exact-completion} we obtain short exact sequences for both the $\EQ$ and $\EW$ bicomplexes. It is now immediate from the snake lemma that we obtain long exact sequences in cohomology:

\[
0\rightarrow \He^0(X,\mathcal U)\rightarrow \He^0\XV\rightarrow \He^0(X,\mathcal W)\rightarrow \He^1(X,\mathcal U)\rightarrow \He^1\XV\rightarrow \He^1(X,\mathcal W)\rightarrow \cdots
\]

\[
0\rightarrow \HQ^0(X,\mathcal U)\rightarrow \HQ^0\XV\rightarrow \HQ^0(X,\mathcal W)\rightarrow \HQ^1(X,\mathcal U)\rightarrow \HQ^1\XV\rightarrow \HQ^1(X,\mathcal W)\rightarrow \cdots
\]

\[
0\rightarrow \HW^0(X,\mathcal U)\rightarrow \HW^0\XV\rightarrow \HW^0(X,\mathcal W)\rightarrow \HW^1(X,\mathcal U)\rightarrow \HW^1\XV\rightarrow \HW^1(X,\mathcal W)\rightarrow \cdots
\]

\end{proof}

As an example we consider the following short exact sequence of $X$-modules, where we are taking the group $G$ to be trivial:

\[
0\rightarrow \ell^1_0(X)\xrightarrow{\iota} \ell^1(X)\xrightarrow{\pi} \C\rightarrow 0
\]

The function spaces are equipped with their usual support functions $\supp(f)=\{x\in X\mid f(x)\not = 0\}$ and $\C$ is equipped with the trivial support function $\supp(\lambda)=\emptyset$ for all $\lambda\in \C$. The map $\iota$ is the standard ``forgetful'' inclusion of $\ell^1_0(X)$ into $\ell^1(X)$ and is easily seen to be a monomorphism. The map $\pi$ is the evaluation of the $\ell^1$ sum and this is an epimorphism. To see this we argue as follows: since the support of any $\lambda\in \C$ is empty it lies within $R$ of any point $x\in X$. We choose the scaled Dirac delta function $\lambda\delta_x\in \ell^1(X)$ which clearly maps to $\lambda$, has norm $|\lambda|$ and $\supp(\lambda\delta_x)=\{x\}$, so putting $M=1$ and $S=0$ satisfies the conditions.

It follows that we obtain a long exact sequence of cohomology:

\[
0\rightarrow \HQ^0(X,\mathcal \ell^1_0)\xrightarrow{\iota_*} \HQ^0(X,\mathcal \ell^1)\xrightarrow{\pi_*} \HQ^0(G,\C )\xrightarrow{D} \HQ^1(X,\mathcal \ell^1_0)\xrightarrow{\iota_*} \HQ^1(X,\mathcal \ell^1)\xrightarrow{\pi_*} \HQ^1(X,\C )\xrightarrow{D} \cdots.
\]

\[
0\rightarrow \HW^0(X,\mathcal \ell^1_0)\xrightarrow{\iota_*} \HW^0(X,\mathcal \ell^1)\xrightarrow{\pi_*} \HW^0(X,\C )\xrightarrow{D} \HW^1(X,\mathcal \ell^1_0)\xrightarrow{\iota_*} \HW^1(X,\mathcal \ell^1)\xrightarrow{\pi_*} \HW^1(X,\C )\xrightarrow{D} \cdots.
\]

\section{A cohomological characterisation of property A}

As an application of the long exact sequence we give our first cohomological characterisation of Yu's property A.

Let $X$ be a metric space, and let $G$ be the trivial group. Recall we have a short exact sequence
\[
0\rightarrow \ell^1_0(X)\xrightarrow{\iota} \ell^1(X)\xrightarrow{\pi} \C\rightarrow 0
\]
where $\C$ is a given the support function where  $\supp(\lambda)$ is empty for all $\lambda$ in $\C$, and $\ell^1_0(X),\ell^1(X)$ are given the usual support functions. Let $\mathbf 1\in \E^{0,-1}(X,\mathbb C)$ denote the constant function 1 on $X$.

\begin{lemma}\label{amazing}
A space $X$ has property $A$ if and only if $\EQ^{0,-1}(X,\ell^1X)$ contains an element $\phi$ such that $D\phi=0$ and $\pi_*\phi=I_Q\mathbf 1$, where $\mathbf 1\in \E^{0,-1}(X,\mathbb C)$ denotes the constant function 1.
\end{lemma}

\begin{proof}
Given a sequence of functions $f_n(x)$ as in the definition of property A, we note that $\phi=f$ has the required properties: The fact that $f_x^n$ is a probability measure ensures that $\pi\phi_n(x)=1$ for all $x,n$, that is $\pi_*\phi=I_Q\mathbf 1$. The other hypotheses of Definition \ref{propAdef} are precisely the assertions that $\phi$ is of controlled supports, and $D\phi=0$ in $\EQ^{1,-1}(X, \ell^1X)$.

Conversely, given an element $\phi\in \EQ^{-1}(X,\ell^1X)$ such that $D\phi=0$ and $\pi_*\phi=I_Q\mathbf 1$, represented by a sequence $\phi_n$, we set $f_n(x)(z)=\frac{|\phi_n(x)(z)|}{\|\phi_n(x)\|_{\ell^1}}$. Since $\pi\phi_n(x)=1$ for all $x,n$ we have $\frac 1{\|\phi_n(x)\|_{\ell^1}}\leq 1$. As an element of $\ell^1(X)$, $f_n(x)$ has the same supports as $\phi_n(x)$, in particular $f$ is of controlled supports. The verification that $\|f_n(x_1) - f_n(x_0)\|_{\ell^1}$ tends to 0 uniformly on $\{ (x_0,x_1)\mid d(x_0,x_1)\leq R\}$ follows from the fact that $D\phi=0$ and the estimate $\frac 1{\|\phi_n(x)\|_{\ell^1}}\leq 1$.
\end{proof}

Note that $\mathbf 1\in \E^{0,-1}(X,\mathbb C)$ is a cocycle. Hence (applying $I_Q$, $I_W$) it represents an element $[\mathbf 1_Q]\in \HQ^0(X, \C)$, and another element $[\mathbf 1_W]\in \HQ^0(X, \C)$.

Recalling the long exact sequences in $\HQ,\HW$
\[
0\rightarrow \HQ^0(X,\mathcal \ell^1_0)\xrightarrow{\iota_*} \HQ^0(X,\mathcal \ell^1)\xrightarrow{\pi_*} \HQ^0(X,\C )\xrightarrow{D} \HQ^1(X,\mathcal \ell^1_0)\xrightarrow{\iota_*} \HQ^1(X,\mathcal \ell^1)\xrightarrow{\pi_*} \HQ^1(X,\C )\xrightarrow{D} \cdots.
\]
\[
0\rightarrow \HW^0(X,\mathcal \ell^1_0)\xrightarrow{\iota_*} \HW^0(X,\mathcal \ell^1)\xrightarrow{\pi_*} \HW^0(X,\C )\xrightarrow{D} \HW^1(X,\mathcal \ell^1_0)\xrightarrow{\iota_*} \HW^1(X,\mathcal \ell^1)\xrightarrow{\pi_*} \HW^1(X,\C )\xrightarrow{D} \cdots.
\]
we have classes $D[\mathbf 1_Q]$ in $\HQ^1(X, \ell^1_0(X))$, and $D[\mathbf 1_W]$ in $\HW^1(X, \ell^1_0(X))$.

We make the following definition.

\begin{defn}
An \emph{asymptotically invariant mean} for $X$ is an element $\mu$ in $\EW^{0,-1}(X,\ell^1(X))$ such that $D\mu=0$ and $\pi_*(\mu)=\mathbf 1_W$.
\end{defn}
Let $\delta$ denote the map $X\to \ell^1(X)$, $x\mapsto \delta_x$. We note that as $\pi_*(\delta)=1$, by exactness we have $\pi_*(\mu)=\mathbf 1_W$ if and only if $\delta-\mu$ lies in the image of $\EW^{0,-1}(X,\ell^1_0(X))$.

We now characterise property A as follows:

\begin{thm} \label{D[1]=0}\label{MainTheorem1}
Let $X$ be a discrete metric space. Then the following are equivalent:

\begin{enumerate}
\item $X$ has property $A$.
\item $[\mathbf 1_Q]\in \Im \pi_*$ in $\HQ^0(X, \C)$.
\item $D[\mathbf 1_Q]=0$ in $\HQ^1(X, \ell^1_0(X))$.
\item $D[\mathbf 1_W]=0$ in $\HW^1(X, \ell^1_0(X))$.
\item $[\mathbf 1_W]\in \Im \pi_*$ in $\HW^0(X, \C)$.
\item $X$ admits an asymptotically invariant mean.
\end{enumerate}
\end{thm}

\begin{proof}
We first show (1)$\iff$(2). Identifying $\HQ^0(X, \C)$ with the kernel of $D:\EQ^{0,-1}(X,\C)\to \EQ^{1,-1}(X,\C)$, we have $[\mathbf 1_Q]\in \Im \pi_*$ if and only if $\EQ^{0,-1}(X,\ell^1X)$ contains an element $\phi$ such that $D\phi=0$ and $\pi_*\phi=I_Q\mathbf 1$. This is equivalent to property A by Lemma \ref{amazing}, hence we have shown (1)$\iff$(2).





Conditions (2) and (3) are equivalent by exactness of the long exact sequence in cohomology, while (3) is equivalent to (4) by Theorem \ref{IQIW}. Conditions (4), (5) are equivalent by a further application of the long exact sequence (this time for the weak-* completion). Finally the equivalence of (5) and (6) is immediate from the definition of asymptotically invariant mean.
\end{proof}

To place this in context we consider the equivariant analog for a group. Let $G$ be a countable group equipped with a left invariant proper metric. Again we have two long exact sequences in cohomology:

\[
0\rightarrow \HQ^0(G,\mathcal \ell^1_0(G))\xrightarrow{\iota_*} \HQ^0(G,\mathcal \ell^1(G))\xrightarrow{\pi_*} \HQ^0(G,\C )\xrightarrow{D} \HQ^1(G,\mathcal \ell^1_0(G))\xrightarrow{\iota_*}   \cdots
\]

\[
0\rightarrow \HW^0(G,\mathcal \ell^1_0(G))\xrightarrow{\iota_*} \HW^0(G,\mathcal \ell^1(G))\xrightarrow{\pi_*} \HW^0(G,\C )\xrightarrow{D} \HW^1(G,\mathcal \ell^1_0(G))\xrightarrow{\iota_*}   \cdots
\]

\begin{thm} \label{amenableD[1]=0}
The following are equivalent:

\begin{enumerate}
\item $G$ is amenable.
\item $[\mathbf 1]\in \Im \pi_*$
\item $D[\mathbf 1_Q]=0$ in $\HQ^1(G, \ell^1_0(G))$.
\item $D[\mathbf 1_W]=0$ in $\HW^1(G, \ell^1_0(G))$.
\end{enumerate}
\end{thm}

\begin{proof}
Suppose that $[\mathbf 1]\in \Im \pi_*$. Identifying $\HQ^0(G,\mathcal \ell^1(G))$ with the kernel of $D:\EQ^{0,-1}\rightarrow \EQ^{1,-1}$ we see that there exists a sequence of equivariant functions  $\phi_n:G \rightarrow\ell^1(G)$ which represents a cocycle and for which $1=\sum\limits_h \phi_n(g)(h)\leq \sum\limits_h |\phi_n(g)(h)| =\|\phi_n(g)\|_{\ell^1}$ for every $g\in X$ and every $n\in \naturals$. Set $f_g^n=|\phi_n(g)|/\|\phi_n(g)\|_{\ell^1}$ to obtain an equivariant element of $\text{Prob}(G)$.

For a given $g,n$ $\supp(f_g^n)=\supp(|\phi_n(g)|/\|\phi_n(g))=\supp(\phi_n(g)$. Since $\phi_n$ is of controlled supports given any  $R>0$ there exists $S>0$ such that $\supp(\phi_n(g))\subseteq B_{S}(g)$. Hence $\supp(f^n_g)\subseteq B_{S}(g)$ as required.

Since $\phi_n$ represents a cocycle, $\|D\phi_n((g_0,g_1))\|_{\ell^1}$ converges to zero uniformly on the set $\{(g_0, g_1)\mid d(g_0, g_1\leq R\}$ so 

\[
\|f^n_{g_1} - f^n_{g_0}\|_{\ell^1}{\rightarrow} 0 \text{ as }{n\rightarrow \infty},
\]
uniformly on the set $\{ (g_0, g_1)\mid d(g_0,g_1)\leq R\}$, hence (2) implies (1). The converse is easy.

Conditions (2) and (3) are equivalent by exactness of the long exact sequence in cohomology, while (3) is equivalent to (4) by Theorem \ref{IQIW}.
\end{proof}

This result should be compared with  the statement and proof of Theorem \ref{unboundedbeauty} parts (1), (2), (3) which invoked the long exact sequence:

\[
0\rightarrow H^0(G,( \ell^\infty G/\mathbb C)^*)\xrightarrow{\pi^*} H^0(G,\ell^\infty G^*)\xrightarrow{\iota^*} H^0(G,\C )\xrightarrow{d} H^1(G, ( \ell^\infty G/\mathbb C)^*)\xrightarrow{\pi^*}   \cdots
\]

induced by the short exact sequence

\[
0\rightarrow \ell^1_0(G)\xrightarrow{\iota} \ell^1(G)\xrightarrow{\pi^*} \C\rightarrow 0
\]

We therefore consider the cohomology groups $\HQW^0(G,\mathcal \ell^1(G)), \HQW^0(G,\C ), \HQW^1(G,\mathcal \ell^1_0(G))$, where $\sim$ denotes $Q$ or $W$, as asymptotic analogs of the groups $H^0(G,\ell^\infty G^*),H^0(G,\C ), H^1(G, ( \ell^\infty G/\mathbb C)^*)$ respectively. 


\skipit{

\section{Transfer of module structures by uniform embeddings}
Here we consider transfer of module structures induced by an equivariant uniform embedding of metric spaces, and in particular establish that the ($G$-equivariant) cohomology of a space is an invariant up to ($G$-equivariant) coarse equivalence.

 Let $f:X\rightarrow Y$ be a $G$-equivariant uniform embedding of metric spaces and let $\mathcal V$ be a $G$-equivariant $Y$-module, $(V, \|\cdot\|, \supp_Y)$. We define the coarse pullback $f^*\mathcal V:=(V, \|\cdot\|, \supp_X)$, by $\supp_X(v)=f^{-1}p\supp_Y(v)$ where $p:2^Y\rightarrow 2^{f(X)}$ is defined by 

\[p(A)=\{y\in f(X)\mid \exists a\in A \text{ with }\lceil d_Y(y,a)\rceil< \lceil d_Y(y',a)\rceil, \forall y'\in f(X)\}.
\]

The map $p$ is a coarse version of the nearest point map; when the metric takes values in $\naturals$, e.g., the edge metric on the vertices of a graph, $p(A)$ is the set of nearest points in $f(X)$ to points in $A$.

It is easy to verify that $f^*\mathcal V$ is an $X$-module. The transfer map $f^*$ also induces a map on cohomology as follows:

Given a cochain $\phi\in E^{p,q}(Y, \mathcal V)$ we obtain an element $f^*\phi\in E^{p,q}(X, f^*\mathcal V)$ by setting $f^*\phi(\mathbf x, \mathbf x', n)=\phi(f(\mathbf x), f(\mathbf x'), n)$. Coarseness of $f$ ensures that for any $R\geq 0$ there exists $S\geq 0$ such that $\|f^*\phi\|_R\leq \|\phi\|_S$ and for all $n$, $\|f^*\phi\|_R^n\leq \|\phi\|_S^n$ . 

Given $\phi\in E^{p,q}(Y, \mathcal V)$ and a proper sequence $R_n$ there exists a proper sequence $S_n$ such that $d(y_i, y_j')\leq R_n$ implies that $\supp_Y\phi(\mathbf y, \mathbf y', n)\subseteq  B_{S_n}(y_i)$ for each $i$. Now consider $f^*\phi$ Given a proper sequence $R_n'$ coarseness of $f$ yields a proper sequence $R_n$ such that $d(x_i,  x_j')\leq R_n'$ implies that $d(f(x_i), f(x_j'))\leq R_n$ so there exists a proper sequence $S_n$ such that $d(x_i, x_j')\leq R_n'$ implies that $\supp_Y\phi(f(\mathbf x), f(\mathbf x'), n)\subseteq  B_{S_n}(f(x_i))$ for each $i$. Hence 

\[
p(\supp_Y\phi(f(\mathbf x), f(\mathbf x'), n))\subseteq  p(B_{S_n}(f(x_i)))\subseteq B_{2S_n+1}(f(x_i)).
\]

As $f$ is a uniform embedding it follows that there is a further proper sequence $S_n'$ such that the preimage of $B_{2S_n+1}(f(x_i))$ is contained in $B_{S_n'}(x_i)$, so $\supp_X(f^*\phi(\mathbf x, \mathbf x', n))\subseteq B_{S_n'}(x_i)$ for each $i$. Hence $f^*\phi$ is of controlled supports.

Hence $f^*\phi\in E^{p,q}(X, f^*\mathcal V)$ as required. Moreover if $\phi\in E^{p,q}_0(Y, \mathcal V)$ then $f^*\phi\in E^{p,q}_0(X, f^*\mathcal V)$

We note that $Df^*\phi=f^*D\phi$ and $df^*\phi=f^*d\phi$ by the usual argument so we obtain a map on bicomplexes inducing maps

\begin{align*}
f^*&:\Ho^*(Y, \mathcal V)\rightarrow \Ho^*(X, f^*\mathcal V)\\
f^*&:\HE^*(Y, \mathcal V)\rightarrow \HE^*(X, f^*\mathcal V)\\
f^*&:\HQ^*(Y, \mathcal V)\rightarrow \HQ^*(X, f^*\mathcal V)\\
\end{align*}

In  the special case where $f$ is the inclusion of a subspace $X$ in $Y$ and the group $G$ is trivial we obtain the well known result:

\begin{lemma}
If $X$ is a subspace of a metric space $Y$ and $Y$ has property A then $X$ has property A.
\end{lemma}

\begin{proof}
The transfer map on cocycles is restriction and maps the Johnson element in $\HQ^1(Y, \ell^1_0(Y))$ to  the Johnson element  in $\HQ^1(X, f^*\ell^1_0(Y))$. If $Y$ has property A then its Johnson element is trivial implying that  its image is trivial in  $\HQ^1(X, f^*\ell^1_0(Y))$. But $f^*\ell^1_0(Y)=\ell^1_0(X)$, so this implies that $X$ has property A.
\end{proof}

}

\section{The asymptotically invariant complex}\label{ai-section}

We pause for a moment to recall the classical definition of bounded cohomology for a group. One first takes the homogeneous bar resolution wherein the $k$-dimensional cochains consist of all bounded functions from $G^{k+1}$ to $\mathbb C$ This cochain complex is exact so has trivial cohomology. This is exhibited by taking a basepoint splitting which is induced by the map $G^k\rightarrow G^{k+1}$ given by inserting the basepoint as an additional (first) co-ordinate. Now one takes the $G$-invariant part of this complex, where $G$ acts diagonally and $\mathbb C$ is equipped with the trivial action of $G$. Since the splitting is not equivariant the corresponding cochain complex is not necessarily exact. When the group $G$ is amenable one can average the splitting over orbits using the invariant mean, and this produces an equivariant splitting which therefore kills all the cohomology in dimensions greater than or equal to $1$.

In this section we will carry out an analogous process for property A and a metric space. Replacing the classical (split) cochain complex by the first row of the $\EQ$ bicomplex, $(\EQ^{0,q},d)$, (which is acyclic since $(\E^{0,q},d)$ is acyclic by Proposition \ref{acyclicrows}) we then take the kernels under the vertical differential $D$ to produce a new cochain complex, which we will call the \emph{asymptotically invariant subcomplex of $\EQ$}. The splitting $s$ of the horizontal differential $d$ does not restrict to this cochain complex leaving room for interesting cohomology. This is the analogue of taking the invariant parts in group cohomology. Similarly $(\EW^{0,q},d)$ is acyclic, but taking the kernels under the vertical differential $D$ we obtain the \emph{asymptotically invariant subcomplex of $\EW$}. We will then show that if the space $X$ has property A, one can asymptotically average the splitting to obtain a splitting of the asymptotically invariant complexes. Hence we deduce (analogously to the case of bounded cohomology) that if $X$ has property $A$ then the cohomologies of both asymptotically invariant complexes vanish in all dimensions greater than 0.

\begin{defn}
We say that an element $\phi$ of $\EQ^{0,q}$ (respectively $\EW^{0,q}$) is \emph{asymptotically invariant} if $D\phi=0$ in $\EQ^{1,q}$  (respectively $\EW^{1,q}$) Let $\EQA^q$, $\EWA^q$, denote the sets of asymptotically invariant elements in $\EQ^{0,q}$ and  $\EW^{0,q}$ respectively. We note as usual that this is defined for $q\geq -1$.
\end{defn}

For notational convenience when considering elements of $\E^{0,q}$ we will write $\phi(x,(y_0,\dots,y_q))$, suppressing the parentheses around the single $x$ variable.

The term asymptotically invariant is motivated by the case of $\EQ^{0,q}$. An element of $\EQ^{0,q}$ is asymptotically invariant if it is represented by a sequence $\phi_n:X\times X^{q+1} \to V$ which is asymptotically invariant in the $x$ variable the following sense. For all $R>0$ the difference $\phi_n(x_1,\mathbf y)-\phi_n(x_0,\mathbf y)$ tends to zero uniformly on $\{((x_0,x_1),\mathbf y) \mid d(x_0,x_1)\leq R, \mathbf y\in X^{q+1}\}$.

We remark that it is essential that we first complete the complex $\E$ and then take the kernels of $D$, not the other way around. If we were to take the kernel of $D:\E^{0,q}\to \E^{1,q}$ we would get functions $\phi(x,(y_0,\dots,y_q))$ which are constant in the $x$ variable, that is, we have invariant rather than asymptotically invariant elements. The kernel of $D:\EQ^{0,q}\to\EQ^{1,q}$ will typically be much larger than the completion of these constant functions, and similarly for $\EW$.

We now make the following elementary observation.

\begin{prop}
The differential $d$ maps  $\EQA^q\XV$ to $\EQA^{q+1}\XV$, and maps $\EWA^q$ to $\EWA^{q+1}\XV$. Hence $(\EQA^q\XV,d)$, $(\EWA^q\XV,d)$ are complexes.
\end{prop}

\begin{proof}
This is immediate from anti-commutativity of the differentials $D,d$.
\end{proof}

Recall that there is a splitting $s:\E^{0,q}\to \E^{0,q-1}$ extending to both generalised completions. Note however that $s$ does not necessarily map either asymptotically invariant complex into itself, as illustrated by the following simple example.

\begin{example}
Let $X$ be the integers equipped with the subspace metric from the reals, and let $\mathcal V$ be the $X$-module $\ell^1\integers$. For elements $x,y_0,y_1\in \integers$ set $\phi(x,(y_0,y_1))=\delta_{y_1}-\delta_{y_0}$, where $\delta_y$ denotes the Dirac function with support $\{y\}$. Clearly $\phi\in \E^{0,1}$, and since $\phi$ is independent of $x$, $D\phi=0$, so $I_Q\phi$ lies in $\EQA^1$, and $I_W\phi$ lies in $\EWA^1$. However 
$$Ds\phi((x_0,x_1),(y_0))=s\phi(x_1,(y_0))-s\phi(x_0,(y_0))=(\delta_{y_0}-\delta_{x_1})-(\delta_{y_0}-\delta_{x_0})=\delta_{x_0}-\delta_{x_1}$$
which has $\ell^1$-norm equal to 2 for all $x_0\neq x_1$. Hence $DsI_Q\phi=I_QDs\phi\neq 0$ and $DsI_W\phi=I_WDs\phi\neq 0$, so neither $sI_Q\phi$ nor $sI_W\phi$ is asymptotically invariant.
\end{example}





\begin{defn}
Let $\HQA^*\XV$ denote the cohomology of the complex $\EQA^*\XV$ and let $\HWA^*\XV$ denote the cohomology of the complex $\EWA^*\XV$
\end{defn}


We will refer to the inclusions of $\EQA^q\XV\hookrightarrow \EQ^{0,q}\XV$ and $\EWA^q\XV\hookrightarrow \EW^{0,q}\XV$ as \emph{augmentation maps}. By analogy with the horizontal case, we will denote these by $D$.

\begin{lemma}
\label{injective-augmentation}
The augmentation maps induce maps on cohomology $\HQA^q\XV\rightarrow \HQ^{q}\XV$ and $\HWA^q\XV\rightarrow \HW^{q}\XV$, which are isomorphisms for $q=0$, and are injective for $q=1$.
\end{lemma}

\begin{proof}
Let $\phi$ be a cocycle in $\EQA^q\XV$. Then $D\phi=0$ since $\phi \in \EQA^q\XV$, and $d\phi=0$ since $\phi$ is a cocycle. Hence including $\EQA^q\XV$ into $\EQ^{0,q}\XV$, the element $\phi$ is a cocycle in the totalisation of the bicomplex, yielding a map $\HQA^q\XV\to \HQ^{q}\XV$. In degree 0 every cocycle is non-trivial, and the condition that $\phi$ is a cocycle is that $D\phi=0$ and $d\phi=0$ in both theories, whence the map is an isomorphism. In degree 1, if $\phi\in \E^{0,1}\XV$ is a coboundary in the totalisation of the bicomplex then there is an element $\psi$ of $\EQ^{0,0}\XV$ such that $(D\oplus d)\psi$ is $(0\oplus \phi)$ in $\EQ^{1,0}\XV\oplus \EQ^{0,1}\XV$. That is $D\psi=0$, so $\psi$ is an element of $\EQA^{0}\XV$, and $d\psi=\phi$. Hence $\phi$ is also a coboundary in $\EQA^1\XV$. Hence the inclusion of $\EQA^1\XV$ into $\EQ^{0,1}\XV$ gives an injection of cohomology.

The proof for $\EW$ is identical.
\end{proof}

Now we restrict to the case where $G$ is trivial, and $\mathcal V$ is $\ell^1_0(X)$, equipped with the $\ell^1$ norm and the usual support function. Consider the Johnson elements $\J^{0,1}(x,(y_0,y_1))=\delta_{y_1}-\delta_{y_0}$ in $\E^{0,1}\XV$ and $\J^{1,0}((x_0,x_1),(y))=\delta_{x_1}-\delta_{x_0}$ in $\E^{1,0}\XV$. Since $D\J^{0,1}=0$ and $d\J^{0,1}=0$ we deduce that $I_Q\J^{0,1}$,  $I_W\J^{0,1}$ give classes in $\HQA^1(\XV)$ and $\HQA^1(\XV)$, which we will denote $[\J_Q^{0,1}]$ and $[\J_W^{0,1}]$.

Applying the augmentation map to $[\J_Q^{0,1}]$ we obtain an element of $\HQ^1\XV$. We will verify that this is cohomologous to $D[\mathbf 1_Q]$, hence the vanishing of $[\J_Q^{0,1}]$ is equivalent to property A. Similarly $[\J_W^{0,1}]$ is cohomologous to $D[\mathbf 1_W]$ and the vanishing of $[\J_Q^{0,1}]$ is equivalent to property A.

\begin{thm}\label{PropA} The augmentation map for $\HQ$ takes $[\J_Q^{0,1}]$ in $\HQA^1(X, \mathcal \ell^1_0(X))$ to $D[\mathbf 1_Q]$ in $\HQ^1(X, \mathcal \ell^1_0(X))$, and the augmentation map for $\HW$ takes the class $[\J_W^{0,1}]$ in $\HWA^1(X, \mathcal \ell^1_0(X))$ to the class $D[\mathbf 1_W]$ in $\HW^1(X, \mathcal \ell^1_0(X))$. Hence the following are equivalent:
\begin{enumerate}
\item $X$ has property $A$.
\item $[\J_Q^{0,1}]=0$ in $\HQA^1(X, \mathcal \ell^1_0(X))$.
\item $[\J_W^{0,1}]=0$ in $\HWA^1(X, \mathcal \ell^1_0(X))$.
\end{enumerate}
\end{thm}

\begin{proof}
The Johnson classes $D[\mathbf 1_Q],D[\mathbf 1_W]$ are given by respectively $I_Q[\J^{1,0}]$ and $I_W[\J^{1,0}]$. We will show that $\J^{1,0}$ is cohomologous to $\J^{0,1}$ in the totalisation of $\E^{*,*}(X, \mathcal \ell^1_0(X))$. From this it is immediate that we have $[\J_Q^{0,1}]=[\J_Q^{1,0}]=D[\mathbf 1_Q]$ in $\HQ^1(X, \mathcal \ell^1_0(X))$ and $[\J_W^{0,1}]=[\J_W^{1,0}]=D[\mathbf 1_W]$ in $\HW^1(X, \mathcal \ell^1_0(X))$.

Let $\phi(x,(y))=\delta_y-\delta_x$. We note that this lies in $\E^{0,0}(X,\ell^1_0(X))$. Computing the coboundaries we have $D\phi((x_0,x_1),(y))=-\delta_{x_1}+\delta_{x_0}=-\J^{1,0}((x_0,x_1),(y))$, while $d\phi(x,(y_0,y_1))=\delta_{y_1}-\delta_{y_0}=\J^{0,1}(x,(y_0,y_1))$. Hence $(D\oplus d)\phi=-\J^{1,0}\oplus \J^{0,1}$, whence $[\J^{1,0}]=[\J^{0,1}]$ in $\He^1(X, \mathcal \ell^1_0(X))$.

By Lemma \ref{injective-augmentation}, $[\J_Q^{0,1}]$ is zero in $\HQA^1(X, \mathcal \ell^1_0(X))$ if and only if it is zero in $\HQ^1(X, \mathcal \ell^1_0(X))$, and we have seen that the latter is $D[\mathbf 1_Q]$. By Theorem \ref{D[1]=0}, this vanishes if and only if $X$ has property A.

Similarly $[\J_W^{0,1}]$ is zero in $\HWA^1(X, \mathcal \ell^1_0(X))$ if and only if it is zero in $\HW^1(X, \mathcal \ell^1_0(X))$, and $[\J_W^{0,1}]=D[\mathbf 1_W]=0$ in $\HW^1(X, \mathcal \ell^1_0(X))$ if and only if $X$ has property A.
\end{proof}

\begin{section} {Vanishing theorems}

We have seen that the map $s$ does not split the coboundary map $d$ in the complexes $\EQA^*,\EWA^*$, however if $X$ has property A then we can use the probability functions given by the definition to asymptotically average the $s\phi$. Having done so we will obtain an actual splitting for the cochain complex, demonstrating the vanishing of the cohomology.

We will make use of the following convolution operator.

\begin{defn}
For $f\in\E^{p,-1}(X,\one(X))$ and $\theta\in \E^{0,q}(X,V)$, define $f*\theta$ by
$$(f*\theta)(\bfx,\mathbf y) = \sum_{z} f(\mathbf x)(z)\theta(z,\bfy).$$
\end{defn}

We make the following estimate:
$$
\|f*\theta\|_R\leq \sup\limits_{\bfx \in \Delta^{p+1}_R, \bfy\in X^{q+1}}\sum\limits_{z\in X}|f(\bfx)(z)|\|\theta(z,\bfy)\|_V\leq \sup\limits_{\bfx\in \Delta_R^{p+1}}\sum\limits_z|f({\bf x},(z))|\|\theta\| = \norm{f}_R\norm{\theta}.
$$
We remark that as $\theta$ lies in the bottom row of the bicomplex, its $R$-norm does not in fact depend on $R$, hence we suppress it from the notation.

This estimate shows that for each $f$ the map $\theta\mapsto f*\theta$ is continuous, and for each $\theta$ the map $f\mapsto f*\theta$ is continuous.



We note that $D(f*\phi)(\bfx,\bfy)=\sum\limits_{z}\sum\limits_{i}(-1)^i f((x_0,\dots,\widehat{x_i},\dots,x_p))(z)\phi(z,\mathbf y)=(Df)*\phi(\bfx,\bfy)$, by exchanging the order of summation.

Similarly $d(f*\phi)(\bfx,\bfy)=(f* d\phi)(\bfx,\bfy)$.

\medskip

The convolution extends in an obvious way to the quotient completion. For $f\in \EQ^{q,-1}(X,\one(X))$, $\phi\in \EQ^{0,q}(X,V)$ we define $f*\phi\in\EQ^{p,q}(X,V)$ by $(f*\phi)_n = f_n*\phi_n$. We note that if either of the sequences $f_n,\phi_n$ tends to 0 as $n\to\infty$, then $(f*\phi)_n$ tends to 0, by the above norm estimate. Hence the convolution is a well defined map
$$\EQ^{p,-1}(X,\one(X)) \times \EQ^{0,q}(X,V)\to \EQ^{p,q}(X,V),$$
i.e.\ as an element of $\EQ^{p,q}(X,V)$, the convolution $f*\phi$ does not depend on the choice of sequences representing $f,\phi$.


Since the convolution is defined term-by-term in $n$, the identities $D(f*\phi)=(Df)*\phi$ and $d(f*\phi)=f*d\phi$ carry over to the quotient completion.

We recall that by Lemma \ref{amazing}, property A is equivalent to the existence of an element $f$ of $\EQ{0,-1}(X,\one(X))$ with $Df=0$ and $\pi_*(f)=I_Q\mathbf 1$. The idea of the proof of the following theorem is that convolving with $f$ allows us to average the splitting $s\phi$, to get an asymptotically invariant element. 

\begin{thm}\label{Triv} If $X$ is a metric space satisfying Yu's property A, then for every $q\geq 1$, $\HQA^q(X, \mathcal V)=0$ for every $X$-module $\mathcal V$. 
\end{thm}

\begin{proof}
Let $\phi\in \EQA^q\XV$ with $q\geq 1$. The element $\phi$ is represented by a sequence $\phi_n$ in $\E^q\XV$ and $s\phi$ is represented by the sequence
\[
s\phi_n(x,(y_0, \ldots, y_{q-1}))= \phi_n(x,(x,y_0, \ldots, y_{q-1})).
\]

Since $D\phi=0$, the sequence $D\phi_n$ tends to zero, that is for all $R>0$, $\|D\phi_n\|_R\to 0$ as $n\to\infty$. By a diagonal argument, if $S_n$ is a sequence tending to infinity sufficiently slowly, then $\|D\phi_n\|_{S_n}\to 0$ as $n\to\infty$. We choose a sequence $S_n$ with this property.

Take an element $f$ in $\EQ^{0,-1}(X,\one(X))$ with $Df=0$ and $\pi_*(f)=I_Q\mathbf 1$, and let $f_n$ be a sequence representing $f$. For simplicity we assume that $\pi(f_n(x))=1$ for all $x,n$; if $f_n'$ is a sequence representing $f$, then $f_n'(x)+(1-\pi(f_n(x)))\delta_x$ also represents $f$, and has sum 1 as required.

By repeating the terms of the sequence $f_n$ we can arrange that $\supp(f_n(x))\subseteq B_{S_n}(x)$ for all $x,n$. Note that our choice of $f$ therefore depends on $S_n$ and hence on $\phi$.

As a remark in passing, we note that taking such a `supersequence' of $f_n$ corresponds in some sense to taking a subsequence of $\phi_n$. If we were working in the classical completion $E_{\mathrm{cs}}/E_0$, then the subsequence would represent the same element of $E_{\mathrm{cs}}/E_0$, however for $\EQ$ this need not be true.

For each $q'$ we now define $s_f:\EQ^{0,q'}(X,V)\to \EQ^{0,q'-1}(X,V)$ by $s_f\psi=f*s\psi$. We first note that for any $\psi$ we have $s_f\psi$ asymptotically invariant. This follows from asymptotic invariance of $f$, since $Ds_f\phi=D(f*s\phi)=(Df)*s\phi=0$. Hence in fact we have a map $s_f:\EQ^{q'}(X,V)\to \EQA^{q'-1}(X,V)$ which restricts to the asymptotically invariant complex.

We claim that for our given $\phi$ we have $(ds_f+s_fd)\phi=\phi$. We have $ds_f\phi=d(f*s\phi)=f*ds\phi$, while $s_fd\phi=f*sd\phi$ by definition. Hence $(ds_f+s_fd)\phi=f*(ds+sd)\phi=f*\phi$ since $ds+sd=1$. It thus remains to show that $f*\phi=\phi$. Notice that since $\sum\limits_{z\in X}f_n(x)(z)=1$ we have $\phi_n(x,\mathbf y)=\sum\limits_{z\in X}f_n(x)(z)\phi_n(x, \mathbf y)$, so we have
\begin{align*}
(f_n*\phi_n-\phi_n)(x,\mathbf y)&=\sum\limits_{z\in X}f_n(x)(z)(\phi_n(z,\mathbf y)-\phi_n(x,\mathbf y))\\
&=\sum\limits_{z\in X}f_n(x)(z)D\phi_n((x,z),\mathbf y).\\
\end{align*}

Taking norms we have $\norm{f_n*\phi_n-\phi_n}\leq \norm{f_n}\norm{D\phi_n}_{S_n}$, since if $d(x,z)>S_n$ then $f_n(x)(z)$ vanishes. We know that $\norm{D\phi_n}_{S_n}\to 0$ as $n\to\infty$, hence we conclude that $f*\phi-\phi=0$ in $\EQA^q\XV$.

We have shown that for every element $\phi \in \EQA^q\XV$ with $q\geq 1$, we can construct maps $s_f:\EQA^{q'}\XV\to \EQA^{q'-1}\XV$ such that $(ds_f+s_fd)\phi=\phi$. (We remark that the choice of $f$, and hence the map $s_f$ depends on the element $\phi$ in question.) It follows that if $\phi$ is a cocycle then $\phi=(ds_f+s_fd)\phi=ds_f\phi$, so every cocycle is a coboundary. Thus we deduce that $\HQA^q(X, \mathcal V)=0$ for $q\geq 1$.
\end{proof}

Combining this with Theorem \ref{PropA} we obtain the following.

\begin{thm}\label{MainTheorem2}
Let $X$ be a discrete metric space. Then the following are equivalent:
\begin{enumerate}
\item $\HQA^q(X, \mathcal V)= 0$ for all $q\geq 1$ and all modules $\mathcal V$ over $X$.
\item $[\J_Q^{0,1}]=0$ in $\HQA^1(X, \mathcal \ell^1_0(X))$.
\item $X$ has property $A$.
\end{enumerate}
\end{thm}

\bigskip

We will now prove a corresponding result for the weak-* completion. The role of $f$ in the above argument will be replaced by an asymptotically invariant mean $\mu$ in $\EW^{0,-1}(X,\one(X))$.

We begin by extending the convolutions to the weak-* completions. First we define $f*\phi$ for $f\in\E^{p,-1}(X,\one(X))$ and $\phi\in \EW^{0,q}(X,V)$. This is defined via its pairing with an element $\alpha$ of $\E^{p,q}(X,V)^*$:
$$\la f*\phi,\alpha \ra=\la \phi,\alpha_f \ra, \text{ where } \la \alpha_f,\theta\ra = \la \alpha,f*\theta\ra, \text { for all } \theta \in  \E^{0,q}(X,V).$$
In other words the operator $\phi\mapsto f*\phi$ on $\EW^{0,q}(X,V)$ is the double dual of the operator $\theta\mapsto f*\theta$ on $\E^{0,q}(X,V)$.

We have $|\la \alpha_f,\theta\ra| \leq \norm{\alpha}_R\norm{f*\theta}_R \leq\norm{\alpha}_R\norm{f}_R\norm{\theta}$ for some $R$ (depending on $\alpha$). Hence for each $\alpha$ there exists $R$ such that
$$|\la f*\phi,\alpha\ra|\leq \norm{\phi}_R\norm{\alpha}_R\norm{f}_R$$
so $f*\phi$ is a continuous linear functional.

We now want to further extend the convolution to define $\eta*\phi$ in $\EW^{p,q}(X,V)$, for $\eta\in\EW^{p,-1}(X,\one(X))$ and $\phi\in \EW^{0,q}(X,V)$. The definition is motivated by the requirement that $(I_Wf)*\phi =  f*\phi$. Hence for $\alpha$ in $\E^{p,q}(X,V)^*$ we will require
$$\la (I_Wf)*\phi,\alpha \ra = \la f*\phi,\alpha\ra.$$
For $\phi\in \EW^{0,q}(X,V)$,  $\alpha \in \E^{p,q}(X,V)^*$, define $\sigma_{\phi,\alpha} \in \E^{p,-1}(X,\one(X))^*$ by 
$$\la \sigma_{\phi,\alpha}, f\ra = \la f*\phi,\alpha\ra=\la \phi,\alpha_f\ra.$$
The above inequalities ensure that $\sigma_{\phi,\alpha}$ is a continuous linear functional.

We observe that $f*\phi$ is determined by the property that $\la f*\phi,\alpha\ra=\la \sigma_{\phi,\alpha},f\ra=\la I_Wf, \sigma_{\phi,\alpha}\ra$. We use this to give the general definition: For $\eta\in\EW^{p,-1}(X,\one(X))$ and $\phi\in \EW^{0,q}(X,V)$, we define $\eta*\phi$ in $\EW^{p,q}(X,V)$ by
$$\la \eta*\phi, \alpha\ra = \la \eta, \sigma_{\phi,\alpha}\ra$$
for all $\alpha$ in $\E^{p,q}(X,V)^*$.

\begin{lemma}\label{identities}
For $\eta\in \EW^{p,-1}(X,\one(X))$ and $\phi\in \EW^{0,q}\XV$ we have $D(\eta*\phi)=(D\eta)*\phi$ and $d(\eta*\phi)=\eta*d\phi$.
\end{lemma}

\begin{proof}
The elements $D(\eta*\phi),d(\eta*\phi)$ are defined by their pairings with respectively $\alpha$ in $\E^{p+1,q}\XV^*$ and $\beta$ in $\E^{p,q+1}\XV^*$. These are given by pairing $\eta$ with respectively $\sigma_{\phi,D^*\alpha}$ and $\sigma_{\phi,d^*\beta}$.

Since for $f\in \E^{p,-1}(X,\one(X))$ we have $\la \sigma_{\phi,D^*\alpha},f\ra=\la \phi, (D^*\alpha)_f\ra$ and $\la \sigma_{\phi,d^*\beta},f\ra=\la \phi, (d^*\beta)_f\ra$, we must determine $(D^*\alpha)_f$ and $(d^*\beta)_f$. Pairing these with an element $\theta$ in $\E^{0,q}\XV$ we have
$$\la (D^*\alpha)_f,\theta\ra = \la\alpha,D(f*\theta)\ra = \la\alpha,(Df)*\theta)\ra, \text{ and } \la (d^*\beta)_f,\theta\ra = \la\beta,d(f*\theta)\ra = \la\beta,f*d\theta\ra.$$
Hence $(D^*\alpha)_f=\alpha_{Df}$ and $(d^*\beta)_f=d^*(\beta_f)$, so we have $\sigma_{\phi,D^*\alpha}=D^*\sigma_{\phi,\alpha}$ and $\sigma_{\phi,d^*\beta}=\sigma_{d\phi,\beta}$. It follows that $D(\eta*\phi)=(D\eta)*\phi$ and $d(\eta*\phi)=\eta*d\phi$ as required.
\end{proof}



Before proceeding with the proof of the vanishing theorem we first establish the following result.

\begin{lemma}\label{TrivialPairing}
If $\eta \in \EW^{0,-1}(X,\ell^1(X))$ is in the image of $\EW^{0,-1}(X,\ell^1_0(X))$, and $\phi\in \EW^{0, q}(X, V)$ with $D\phi=0$ then $\eta*\phi=0$. 
\end{lemma}

\begin{proof}
The statement that $\eta*\phi=0$, amounts to the assertion that $\la \eta, \sigma_{\phi,\alpha}\ra=0$ for all $\alpha$ in $\E^{0,q}(X,V)^*$. Since the image of $I_W$ is dense in $\EW^{0,-1}(X,\ell^1_0(X))$ in the \weakstar topology, it suffices to show that $\langle \sigma_{\phi,\alpha}, f\rangle = 0$ for all $f\in \E^{0,-1}(X,\ell^1_0(X))$. We note that
$$\langle \sigma_{\phi,\alpha}, f\rangle = \la f*\phi,\alpha\ra = \la \phi, \alpha_f \ra.$$
We will show that $\alpha_f$ is a `boundary,' that is $\alpha_f$ is in the range of $D^*$. As $D\phi=0$ it will follow that the pairing is trivial.

We define a boundary map $\partial:\one(X\times X)\to \one_0(X)$ by $(\partial H)(z_0)=\sum\limits_{z\in X} H(z_1,z_0)-H(z_0,z_1)$. Equivalently, we can write $\partial F=\sum\limits_{z_0,z_1\in X} H(z_0,z_1)(\delta_{z_1}-\delta_{z_0})$.

We note that $\partial$ is surjective: For $h\in \one_0(X)$ and $x$ in $X$, let $H(z_0,z_1)=h(z_1)$ if $z_0=x,z_1\neq x$ and let $H(z_0,z_1)=0$ otherwise. Then $\partial H=h$. We note that $\norm{H}\leq \norm{h}$, and $\supp(H)\subseteq \{x\}\times \supp(h)$. For each $x$, let $F(x)$ be the lift of $f(x)$ constructed in this way, so that $\norm{F(x)}\leq \norm{f(x)}$ for all $x$, and as $f$ is of controlled supports there exists $R$ such that if $F(x)(z_0,z_1)\neq 0$ then $z_0,z_1\in B_R(x)$.



Writing $(\partial F)(x)=\partial(F(x))$, for $\theta\in \E^{0, q}(X, V)$, we have
$$\la \alpha_f,\theta \ra = \la \alpha, f*\theta\ra = \la \alpha, (\partial F)*\theta\ra.$$

Now compute $(\partial F)*\theta$. We have
\begin{align*}
((\partial F)*\theta)(x,\bfy)&=\sum_z \partial F(x)(z)\theta(z,\bfy)
=\sum_{z,z_0,z_1} F(x)(z_0,z_1)(\delta_{z_1}(z)-\delta_{z_0}(z))\theta(z,\bfy)\\
&=\sum_{z_0,z_1}F(x)(z_0,z_1)D\theta((z_0,z_1),\bfy)
\end{align*}

We define $T_F:\E^{1, q}(X, V)\to \E^{0, q}(X, V)$ by $(T_F\zeta)(x,\mathbf y)= \sum\limits_{z_0,z_1}F(x)(z_0,z_1)\zeta((z_0,z_1),\bfy)$. As $F(x)(z_0,z_1)\neq 0$ implies $z_0,z_1$ lie in the ball $B_R(x)$, we have the estimate
$$
\|T_F\zeta\|\leq \sup\limits_{x\in X, \bfy\in X^{q+1}}\sum\limits_{z_0,z_1\in X}|F(x)(z_0,z_1)|\|\zeta((z_0,z_1),\bfy)\|_V\leq \sup_{x\in X}\norm{F(x)}\norm{\zeta}_R\leq \norm{f}\norm{\zeta}_R.
$$
hence $T_F$ is continuous.

We conclude that
$$\la \alpha_f,\theta\ra = \la \alpha, (\partial F)*\theta\ra = \la \alpha, T_FD\theta\ra=\la D^*T_F^*\alpha,\theta\ra$$
for all $\theta$, hence $\alpha_f=D^*T_F^*\alpha$, so that
$$\la \phi, \alpha_f\ra=\la \phi, D^*T_F^*\alpha\ra=\la D\phi, T_F^*\alpha\ra=0.$$
This completes the proof.
\end{proof}
%
%
%

We now prove the vanishing theorem.

\begin{thm}\label{TrivW} If $X$ is a metric space satisfying Yu's property A, then for every $q\geq 1$, $\HWA^q(X, \mathcal V)=0$ for every $X$-module $\mathcal V$.

Specifically, if $\mu$ is an asymptotically invariant mean then $s_\mu\phi=\mu*s\phi$ defines a splitting of the asymptotically invariant complex.
\end{thm}

\begin{proof}
By Theorem \ref{D[1]=0}, property A guarantees the existence of an asymptotically invariant mean $\mu$, that is an element $\mu$ in $\EWA^{0,-1}$ such that $\pi_*(\mu)=0$.

We define $s_\mu:\EW^{0,q}\XV\to\EW^{0,q-1}\XV$ by $s_\mu\phi=\mu*s\phi$. By Lemma \ref{identities} we have $Ds_\mu\phi=D(\mu*s\phi)=(D\mu)*s\phi$. Since $\mu$ is asymptotically invariant $D\mu=0$, so $s_\mu\phi$ is also asymptotically invariant. Hence $s_\mu$ restricts to a map $s_\mu:\EWA^{0,q}\XV\to\EWA^{0,q-1}\XV$.
We must now verify that $s_\mu$ is a splitting. By Lemma \ref{identities}, and using the fact that $ds+sd=1$ we have
$$(ds_\mu+s_\mu d)\phi = d(\mu*s\phi)+\mu*sd\phi=\mu*ds\phi+\mu*sd\phi=\mu*\phi.$$
It thus remains to show that $\mu*\phi=\phi$.

Let $\delta$ denote the map $X\to \one(x), x\mapsto \delta_x$. We have $\pi_*(I_W \delta)=1=\pi_*(\mu)$, so for $\eta=\delta-\mu$ we have $\pi_*(\eta)=0$. Hence $\eta$ is in the image of $\EW^{0,-1}(X,\ell^1_0(X))$. As $D\phi=0$, it follows from Lemma \ref{TrivialPairing} that $\eta*\phi=0$. Thus $\mu*\phi=(I_W\delta)*\phi=\delta*\phi$. It is easy to see that convolution with $\delta$ yields the identity map on $\E^{0,q}\XV$, hence its double dual is again the identity map. Thus $\mu*\phi=\delta*\phi=\phi$ as required.

This completes the proof.
\end{proof}

Combining this with Theorem \ref{PropA} we obtain the following.

\begin{thm}\label{MainTheorem3}
Let $X$ be a discrete metric space. Then the following are equivalent:
\begin{enumerate}
\item $\HWA^q(X, \mathcal V)= 0$ for all $q\geq 1$ and all modules $\mathcal V$ over $X$.
\item $[\J_W^{0,1}]=0$ in $\HWA^1(X, \mathcal \ell^1_0(X))$.
\item $X$ has property $A$.
\end{enumerate}
\end{thm}

\end{section}

\skipit{

\section{Amenability}
\note{EDIT THIS TEXT-By invoking equivariance throughout, in definitions and assertions, where this makes sense, we obtain a variation on Johnson's theorem as follows, concerning the characterisation of amenability in terms of the vanishing of bounded cohomology. Where Johnson's theorem proceeds in large part by demonstrating that vanishing of a certain cocycle yields the existence of an invariant mean on the group, vanishing in our theory yields a F\o lner sequence.}

Now recall that given a group $G$ acting by isometries on a metric space $X$ an $X$-module $\mathcal V=(V, \|\cdot\|, \supp)$ is said to be $G$-equivariant if it is equipped with an isometric action of $G$ by bounded linear maps, such that $\supp(gv)=g\supp(v)$ for all $g\in G$ and all $v\in V$. The action of $G$ extends in the usual way to a diagonal action on  the cochain spaces in the bicomplex $E^{p,q}(X, \mathcal V)$ defined by $(g\cdot\phi)(\mathbf x, \mathbf y, n)=g(\phi(g^{-1}(\mathbf x), g^{-1}(\mathbf y), n))$. The action commutes with the differentials and preserves the subcomplex $E^{p,q}_0(X, \mathcal V)$. The $G$-invariant part, which comprises the $G$-equivariant cochains is preserved by the differentials in both bicomplexes. We now take the quotient complex in which we factor out the $G$-invariant asymptotically vanishing cochains in the $G$-invariant cochains. As in the non-equivariant case the rows of the bicomplex are acyclic since the splitting given in that argument is $G$-equivariant. The spaces ${E^{p,-1}(X, \mathcal V)^G\over E_0^{p,-1}(X, \mathcal V)^G}$ are the augmentation of this complex and equipping these with the restriction of $D$ they form a cochain complex whose cohomology is isomorphic to the cohomology of the totalisation.

As in the non-equivariant case we also consider the augmentation of the columns given by the kernels of 

\[
D:{E^{0,q}(X, \mathcal V)^G\over E_0^{0,q}(X, \mathcal V)^G}\rightarrow {E^{1,q}(X, \mathcal V)^G\over E_0^{1,q}(X, \mathcal V)^G}.
\]

The horizontal differential  $d$ restricts to these kernels hence we obtain a new cochain complex. The  cohomology groups are denoted $H^q_B(X, V)$.

\begin{remark}
If $X$ is equipped with two coarsely equivalent metrics and $G$ acts by isometries on both then for any $G$-equivariant module $\mathcal V$ over $X$  the equivariant bicomplexes we obtain from each metric are identical. This follows from the corresponding remark in the non-equivariant case.  In particular  taking $X=G$  to b a countable discrete group acting on itself by left multiplication and equipped with a proper left invariant metric  we observe that invoking uniqueness of the metric up to coarse equivalence that the cohomology does not depend on the metric and we denote it by $H^q_B(G, V)$.
\end{remark}

\begin{thm}\label{MainTheorem2}
Let $G$ be a discrete  group. The following are equivalent:
\begin{enumerate}
\item $H^q_B(G, V)= 0$ for all $q\geq 1$ and all $G$-equivariant modules $V$ over $G$.
\item  $H^1_B(G, \ell^1_0G)= 0$.
\item The class $[J]=[\delta_{y_1}-\delta_{y_0}]\in H^1_B(G, \ell^1_0G)$ is $\text{zero}$.
\item $G$ is amenable.
\end{enumerate}
\end{thm}

\begin{proof}
$(1)\implies (2)\implies( 3)$ is trivial.

$(3)\implies (4)$: As in  the proof of Theorem ?? part ??? the cohomology $H^1_B(G, \ell^1_0(G))$ maps to the cohomology of the totalisation of the bicomplex, which we identify with the cohomology of the complex $\left ({E^{p,-1}(G, \ell^1_0(G))^G\over E^{p,-1}_0(G, \ell^1_0(G))^G}, D\right)$

The image of $[J]$ is cohomologous to $D[\mathbf 1]$. The latter vanishes if and only if $\mathbf 1$ is  the image of a cocycle in ${E^{0,-1}(G, \ell^1(G))^G\over E^{0,-1}_0(G, \ell^1(G))^G}$. Such an element is represented by a $G$-equivariant function $\phi:G\times \naturals\rightarrow \ell^1(G)$ such that for any $R\geq 0$ the norm $\|\phi(g,n)-\phi(g',n)\|_{\ell^1}\to 0$ as $n\to \infty$ uniformly on those $g,g'$ with $d(g,g')\leq R$. Note that for all $g\in G$ we have $\|g\phi(e,n)-\phi(e,n)\|_{\ell^1}=\|\phi(g,n)-\phi(e,n)\|_{\ell^1}$ which tends to $0$ as $n\to \infty$ so setting $f^n:={|\phi(e,n)|\over\|\phi(e,n)\|_{\ell^1}}$ we obtain a Reiter sequence for the group.

$(4)\implies(1)$. We follow the same argument as in Theorem \ref{Triv}. Given a Reiter sequence $f^n$, let $f^n_g=gf^n$. Using this equivariant family to construct a splitting of the differential $d$ as in the proof of Theorem \ref{Triv}, the splitting we produce is now equivariant. Hence  $H^q_B(G, V)$ vanishes for all $q\geq 1$.
\end{proof}

}

\skipit{

\appendix
\section{The Brooks cocycle}
Let $F_2$ be the free group of rank $2$ generated by elements $a$ and $b$ and choose a cyclically reduced word $\omega$ in this basis. Note that $\omega^{-1}$ is also a cyclically reduced word. Given an element $g\in F_2$ we write it uniquely as a reduced word in $a, b$ and define $|g|_\omega$ to be the number of times the reduced word $\omega$ appears in the reduced word representation of $g$ minus the number of times $\omega^{-1}$ appears. This function is unbounded since $|\omega^n|_\omega\geq n$. (It is not actually necessary to choose a cyclically reduced word in this construction, however it does simplify this part of the argument.) This function defines a $1$-cochain $\phi_\omega(g,h)=|g^{-1}h|_\omega$ and we take its coboundary 

\[
d\phi_\omega(g,h,k)=|h^{-1}k|_\omega-|g^{-1}k|_\omega+|g^{-1}h|_\omega.
\] 

While the function $|\cdot|_\omega$ is unbounded $d\phi_\omega$ is bounded. Since $d\phi_\omega$ is obtained as the image of a cochain it is a coboundary, and therefore necessarily a cocycle, . Since it is bounded it represents an element of the second bounded cohomology of $F_2$ Brooks demonstrated that such elements are non-trivial by regarding $|\cdot |_\omega$  as a quasi-homomorphism and arguing that there are no non-trivial bounded quasi-homomorphisms from $F_2$ to the reals. Here we show explicitly that the cocycle $d\phi_{ab}$  is non-trivial by pairing it with an element of the homology group $H_2^{\ell^1}(F_2, \mathbb C)$ by analogy with our argument for non-triviality of the Johnson element for this group.

The first step is to construct the required cycle. First we define $\ell^1$ chains where for $g,x,y,z\in F_2$, $g(x,y,z)=(gx,gy,gz)$:

$c_{ab}=\sum\limits_{g\in G} g(e,a,ab)+\frac12 g(e, ab, (ab)^2)+\frac14 g(e,
(ab)^2, (ab)^4) + \frac18  g(e, (ab)^4, (ab)^8) + \ldots$

$c_{a}=\sum\limits_{g\in G} g(e,a,a^2)+\frac12 g(e, a^2, a^4)+\frac14 g(e,
a^4, a^8) + \frac18  g(e, a^8, a^16) + \ldots$

$c_{b}=\sum\limits_{g\in G} g(e,b,b^2)+\frac12 g(e, b^2, b^4)+\frac14 g(e,
b^4, b^8) + \frac18  g(e, b^8, b^16) + \ldots$

$c_{ba}=\sum\limits_{g\in G} g(e,b,ba)+\frac12 g(e, ba, (ba)^2)+\frac14 g(e,(ba)^2, (ba)^4) +
\frac18  g(e, (ba)^4, (ba)^8) + \ldots$

Next we take the element  $C_{ab}=c_{ab}+c_{ba}-(c_a+c_b)$. We claim that this defines an element of $H_2^{\ell^1}(F_2, \mathbb C)$. To see this compute the boundary of $c_{ab}+c_{ba}$, and  $c_a + c_b$ and notice that in each case we get the constant function $\mathbf 2$ on the edges, so the difference is a cocycle as required. (We could halve the coefficients of $c_{ab}+c_{ba}-(c_a+c_b)$ so that the boundaries of both $c_{ab}+c_{ba}$ and $c_a + c_b$ are the fundamental class of the Cayley graph, but doubling the cycle like this makes the numbers a little easier.) 

Finally we compute the pairing of the Brooks cocycle $d\phi_{ab}$ with our cycle $C_{a,b}$, using linearity.

For the $c_{ab}$ pairing, all terms cancel except the first one:
$f_{ab}(e,a,ab)=|a^{-1}ab|_{ab}-|ab|_{ab}+|a|_{ab}=-1$ while
$f_{ab}(e,ab,(ab)^2)=|(ab)^{-1}(ab)^2|_{ab}-|(ab)^2|_{ab}+|ab|_{ab}=0$ etc.

For $c_{ba}$ the pairing is -1 again, but this time computed as an infinite sum:
\begin{align*}
d\phi_{ab}(e,b,ba)&=0\\
d\phi_{ab}(e,ba,(ba)^2)&=|ba|_{ab}-|baba|_{ab}+|ba|_{ab}=-1\\
d\phi_{ab}(e,(ba)^2,(ba)^4)&=|baba|_{ab}-|babababa|_{ab}+|baba|_{ab}=1-3+1=-1\\
\vdots\\
\end{align*}

Adding these up with the coefficients 1,1/2,1/4, etc, we get a total of -1.

Hence the pairing of the Brooks cycle with $c_{ab}+c_{ba}-(c_a+c_b)$ is $-2$ demonstrating that both the Brooks cocycle and our cycle are non-trivial as required.

}

\end{document}